\documentclass[english,smallextended]{svjour3}
\usepackage[T1]{fontenc}
\usepackage[latin9]{inputenc}
\usepackage{float}
\usepackage{url}
\usepackage{amsmath}
\usepackage{amssymb}
\usepackage{graphicx}

\makeatletter

\providecommand{\tabularnewline}{\\}
\floatstyle{ruled}
\newfloat{algorithm}{tbp}{loa}
\providecommand{\algorithmname}{Algorithm}
\floatname{algorithm}{\protect\algorithmname}

\newtheorem{assume}{Assumption}

\usepackage{algpseudocode}
\usepackage{matlab-prettifier}
\usepackage{listings} 
\smartqed

\@ifundefined{showcaptionsetup}{}{%
 \PassOptionsToPackage{caption=false}{subfig}}
\usepackage{subfig}
\makeatother

\usepackage{babel}
\usepackage{listings}

\begin{document}
\title{Sparse Semidefinite Programs with Guaranteed Near-Linear Time Complexity
via Dualized Clique Tree Conversion\thanks{This work was supported by the ONR YIP Award, DARPA YFA Award, AFOSR
YIP Award, NSF CAREER Award, and ONR N000141712933.}}
\titlerunning{Sparse SDPs with Guaranteed Near-Linear Time Complexity via Dualized
CTC}
\author{Richard Y. Zhang \and Javad Lavaei}
\institute{R. Y. Zhang \at Dept. of Industrial Engineering and Operations Research\\
University of California, Berkeley\\
Berkeley, CA 94720, USA\\
\emph{Present address:} Dept. of Electrical and Computer Engineering\\
University of Illinois at Urbana-Champaign\\
306 N Wright St, Urbana, IL 61801\\
\email{ryz@illinois.edu}\\
\and J. Lavaei \at Dept. of Industrial Engineering and Operations
Research\\
University of California, Berkeley\\
Berkeley, CA 94720, USA\\
\email{lavaei@berkeley.edu}}
\date{~}
\maketitle
\begin{abstract}
Clique tree conversion solves large-scale semidefinite programs by
splitting an $n\times n$ matrix variable into up to $n$ smaller
matrix variables, each representing a principal submatrix of up to
$\omega\times\omega$. Its fundamental weakness is the need to introduce
\emph{overlap constraints} that enforce agreement between different
matrix variables, because these can result in dense coupling. In this
paper, we show that by dualizing the clique tree conversion, the coupling
due to the overlap constraints is guaranteed to be sparse over dense
blocks, with a block sparsity pattern that coincides with the adjacency
matrix of a tree. We consider two classes of semidefinite programs
with favorable sparsity patterns that encompass the MAXCUT and MAX
$k$-CUT relaxations, the Lovasz Theta problem, and the AC optimal
power flow relaxation. Assuming that $\omega\ll n$, we prove that
the per-iteration cost of an interior-point method is \emph{linear}
$O(n)$ time and memory, so an $\epsilon$-accurate and $\epsilon$-feasible
iterate is obtained after $O(\sqrt{n}\log(1/\epsilon))$ iterations
in \emph{near-linear} $O(n^{1.5}\log(1/\epsilon))$ time. We confirm
our theoretical insights with numerical results on semidefinite programs
as large as $n=13659$.
\end{abstract}

\section{Introduction}

\global\long\def\R{\mathbb{R}}%
\global\long\def\S{\mathbb{S}}%
\global\long\def\tr{\mathrm{tr}\,}%
\global\long\def\T{\mathcal{T}}%
\global\long\def\varL{\mathcal{L}}%
\global\long\def\J{\mathcal{J}}%
\global\long\def\Set{\mathcal{S}}%
\global\long\def\Edges{\mathcal{E}}%
\global\long\def\ch{\mathrm{ch}}%
\global\long\def\p{\mathrm{p}}%
\global\long\def\K{\mathcal{K}}%
\global\long\def\A{\mathbf{A}}%
\global\long\def\B{\mathbf{B}}%
\global\long\def\N{\mathbf{N}}%
\global\long\def\NN{\mathcal{N}}%
\global\long\def\L{\mathbf{L}}%
\global\long\def\P{\mathbf{P}}%
\global\long\def\M{\mathbf{M}}%
\global\long\def\H{\mathbf{H}}%
\global\long\def\b{\mathbf{b}}%
\global\long\def\f{\mathbf{f}}%
\global\long\def\q{\mathbf{q}}%
\global\long\def\c{\mathbf{c}}%
\global\long\def\C{\mathcal{C}}%
\global\long\def\vector{\mathrm{svec}\,}%
\global\long\def\col{\mathrm{col}\,}%
\global\long\def\row{\mathrm{row}\,}%
\global\long\def\nnz{\mathrm{nnz}\,}%
\global\long\def\D{\mathbf{D}}%
\global\long\def\diag{\mathrm{diag}}%
\global\long\def\one{\mathbf{1}}%
\global\long\def\tw{\mathrm{tw}}%
\global\long\def\wid{\mathrm{wid}}%
\global\long\def\Int{\mathrm{Int}}%
Given $n\times n$ real symmetric matrices $C,A_{1},\ldots,A_{m}$
and real scalars $b_{1},\ldots,b_{m}$, we consider the standard form
semidefinite program 
\begin{alignat}{2}
 & \text{minimize } & C\bullet X & \tag{SDP}\label{eq:SDP}\\
 & \text{subject to } & A_{i}\bullet X & =b_{i}\quad\text{for all }i\in\{1,\ldots,m\}\nonumber \\
 &  & X & \succeq0,\nonumber 
\end{alignat}
over the $n\times n$ real symmetric matrix variable $X$. Here, $A_{i}\bullet X=\tr A_{i}X$
refers to the usual matrix inner product, and $X\succeq0$ restricts
to be symmetric positive semidefinite. Instances of (\ref{eq:SDP})
arise as some of the best convex relaxations to nonconvex problems
like graph optimization~\cite{lovasz1979shannon,goemans1995improved},
integer programming~\cite{sherali1990hierarchy,lovasz1991cones,lasserre2001explicit,laurent2003comparison},
and polynomial optimization~\cite{lasserre2001global,parrilo2000structured}.

Interior-point methods are the most reliable approach for solving
small- and medium-scale instances of (\ref{eq:SDP}), but become prohibitively
time- and memory-intensive for large-scale instances. A fundamental
difficulty is the constraint $X\succeq0$, which densely couples all
$O(n^{2})$ elements within the matrix variable $X$ to each other.
The linear system solved at each iteration, known as the \emph{normal
equation} or the \emph{Schur complement equation}, is usually fully-dense,
irrespective of sparsity in the data matrices $C,A_{1},\ldots,A_{m}$.
With a naïve implementation, the per-iteration cost of an interior-point
method is roughly the same for highly sparse semidefinite programs
as it is for fully-dense ones of the same dimensions: at least cubic
$(n+m)^{3}$ time and quadratic $(n+m)^{2}$ memory. (See e.g.~Nesterov~\cite[Section~4.3.3]{nesterov2013introductory}
for a derivation.)

Much larger instances of (\ref{eq:SDP}) can be solved using the \emph{clique
tree conversion} technique of Fukuda et al.~\cite{fukuda2001exploiting}.
The main idea is to use an interior-point method to solve a reformulation
whose matrix variables $X_{1},\ldots,X_{n}\succeq0$ represent principal
submatrices of the original matrix variable $X\succeq0$, as in\footnote{Throughout this paper, we denote the $(i,j)$-th element of the matrix
$X$ as $X[i,j]$, and the submatrix of $X$ formed by the rows in
$I$ and columns in $J$ as $X[I,J]$.}
\begin{equation}
X_{j}\equiv X[J_{j},J_{j}]\succeq0\qquad\text{ for all }j\in\{1,\ldots,n\}\label{eq:relax}
\end{equation}
where $J_{1},J_{2},\dots,J_{n}\subseteq\{1,2,\dots,n\}$ denote row/column
indices, and to use its solution to recover a solution to the original
problem in closed-form. Here, different $X_{i}$ and $X_{j}$ interact
only through the linear constraints 
\begin{equation}
A_{i}\bullet X=A_{i,1}\bullet X_{1}+\cdots+A_{i,n}\bullet X_{n}=b_{i}\qquad\text{ for all }i\in\{1,\ldots,m\},\label{eq:lincon}
\end{equation}
and the need for their overlapping elements to agree,
\begin{equation}
X_{i}[\alpha,\beta]=X_{j}[\alpha',\beta']\qquad\text{ for all }J_{i}(\alpha)=J_{j}(\alpha'),\quad J_{i}(\beta)=J_{j}(\beta').\label{eq:overlap}
\end{equation}
As a consequence, the normal equation associated with the reformulation
is often \emph{block }sparse\textemdash sparse over fully-dense blocks.
When the maximum order of the submatrices
\begin{equation}
\omega=\max\{|J_{1}|,|J_{2}|,\ldots,|J_{n}|\}\label{eq:omega_def-1}
\end{equation}
is significantly smaller than $n$, the number of linearly independent
constraints is bounded\footnote{The symmetric matrices $A_{1},A_{2},\dots,A_{m}$ share an aggregate
sparsity pattern $E$ that contains at most $\omega n$ nonzero elements
(in the lower-triangular part). The set of symmetric matrices with
sparsity pattern $E$ is a linear subspace of $\R^{n\times n}$, with
rank at most $\omega n$. Therefore, the number of linearly independent
$A_{1},A_{2},\dots,A_{m}$ is at most $\omega n$. } $m\le\omega n$, and the per-iteration cost of an interior-point
method scales as low as \emph{linearly} with respect to $n+m$. This
is a remarkable speed-up over a direct interior-point solution of
(\ref{eq:SDP}), particularly in view of the fact that the original
matrix variable $X\succeq0$ already contains more than $n^{2}/2$
degrees of freedom on its own. 

In practice, clique tree conversion has successfully solved large-scale
instances of (\ref{eq:SDP}) with $n$ as large as tens of thousands~\cite{molzahn2013implementation,madani2015convex,madani2016promises,Eltved2019}.
Where applicable, the empirical time complexity is often as low as
linear $O(n+m)$. However, this speed-up is not guaranteed, not even
on highly sparse instances of (\ref{eq:SDP}). We give an example
in Section~\ref{sec:Cost-of-an} whose data matrices $A_{1},\ldots,A_{m}$
each contains just a single nonzero element, and show that it nevertheless
requires at least $(n+m)^{3}$ time and $(n+m)^{2}$ memory to solve
using clique tree conversion. 

The core issue, and indeed the main weakness of clique tree conversion,
is the overlap constraints (\ref{eq:overlap}), which are imposed
in addition to the constraints (\ref{eq:lincon}) already present
in the original problem~\cite[Section 14.2]{vandenberghe2015chordal}.
These overlap constraints can significantly increase the size of the
normal equation solved at each interior-point iteration, thereby offsetting
the benefits of increased sparsity~\cite{sun2014decomposition}.
In fact, they may contribute more nonzeros to the normal matrix of
the converted problem than contained in the fully-dense normal matrix
of the original problem. In~\cite{andersen2014reduced}, omitting
some of the overlap constraints made the converted problem easier
to solve, but at the cost of also making the reformulation from (\ref{eq:SDP})
inexact. 

\subsection{Contributions}

In this paper, we show that the density of the overlap constraints
can be fully addressed using the \emph{dualization} technique of Löfberg~\cite{lofberg2009dualize}.
By dualizing the reformulation generated by clique tree conversion,
the overlap constraints are guaranteed to contribute $O(\omega^{4}n)$
nonzero elements to the normal matrix. Moreover, these nonzero elements
appear with a block sparsity pattern that coincides with the adjacency
matrix of a tree. Under suitable assumptions on the original constraints
(\ref{eq:lincon}), this favorable block sparsity pattern allows us
to \emph{guarantee} an interior-point method per-iteration cost of
$O(\omega^{6}n)$ time and memory, by using a specific fill-reducing
permutation in computing the Cholesky factor of the normal matrix.
After $O(\sqrt{\omega n}\log(1/\epsilon))$ iterations, we arrive
at an $\epsilon$-accurate solution of (\ref{eq:SDP}) in near-linear
$O(\omega^{6.5}n^{1.5}\log(1/\epsilon))$ time.

Our first main result guarantees these complexity figures for a class
of semidefinite programs that we call \emph{partially separable semidefinite
programs}. Our notion is an extension of the partially separable cones
introduced by Sun, Andersen, and Vandenberghe~\cite{sun2014decomposition},
based in turn on the notion of partial separability due to Griewank
and Toint~\cite{griewank1982partitioned}. We show that if an instance
of (\ref{eq:SDP}) is partially separable, then an optimally sparse
clique tree conversion reformulation can be constructed in $O(\omega^{3}n)$
time, and then solved using an interior-point method to $\epsilon$-accuracy
in $O(\omega^{6.5}n^{1.5}\log(1/\epsilon))$ time. Afterwards, a corresponding
$\epsilon$-accurate solution to (\ref{eq:SDP}) is recovered in $O(\omega^{3}n)$
time, for a complete end-to-end cost of $O(\omega^{6.5}n^{1.5}\log(1/\epsilon))$
time. 

Semidefinite programs that are not partially separable can be systematically
``separated'' by introducing auxiliary variables, at the cost of
increasing the number of variables that must be optimized. For a class
of semidefinite programs that we call \emph{network flow semidefinite
programs}, the number of auxiliary variables can be bounded in closed-form.
This insight allows us to prove our second main result, which guarantees
the near-linear time figure for network flow semidefinite programs
on graphs with small degrees and treewidth.

\subsection{Comparisons to prior work}

At the time of writing, clique tree conversion is primarily used as
a preprocessor for an off-the-shelf interior-point method, like SeDuMi
and MOSEK. It is often implemented using a parser like CVX~\cite{andersen2013cvxopt}
and YALMIP~\cite{lofberg2004yalmip} that converts mathematical expressions
into a compatible data format for the solver, but this process is
very slow, and usually destroys the inherent structure in the problem.
Solver-specific implementations of clique tree conversion like SparseColo~\cite{fujisawa2009user,kim2011exploiting}
and OPFSDR~\cite{andersen2018opfsdr} are much faster while also
preserving the structure of the problem for the solver. Nevertheless,
the off-the-shelf solver is itself structure-agnostic, so an improved
complexity figure cannot be guaranteed. 

In the existing literature, solvers designed specifically for clique
tree conversion are generally first-order methods~\cite{sun2014decomposition,madani2015admm,zheng2019chordal}.
While their per-iteration cost is often linear time and memory, they
require up to $O(1/\epsilon)$ iterations to achieve $\epsilon$-accuracy,
which is exponentially worse than the $O(\log(1/\epsilon))$ figure
of interior-point methods. While it is possible to incorporate a first-order
method within an outer interior-point iteration~\cite{annergren2014distributed,pakazad2014distributed,zhang2018gmres},
this does not improve upon the $O(1/\epsilon)$ iteration bound, because
the first-order method solves an increasingly ill-conditioned subproblem,
with condition number that scales $O(1/\epsilon^{2})$ for $\epsilon$-accuracy.

Andersen, Dahl, and Vandenberghe~\cite{andersen2010implementation}
describe an interior-point method that exploits the same chordal sparsity
structure that underlies clique tree conversion, with a per-iteration
cost of $O(\omega^{3}nm+\omega m^{2}n+m^{3})$ time. The algorithm
solves instances of (\ref{eq:SDP}) with a small number of constraints
$m=O(1)$ in near-linear $O(\omega^{3}n^{1.5}\log(1/\epsilon))$ time.
However, substituting $m\le\omega n$ yields a general time complexity
figure of $O(\omega^{3}n^{3.5}\log(1/\epsilon))$, which is comparable
to the cubic time complexity of a direct interior-point solution of
(\ref{eq:SDP}).

In this paper, we show that off-the-shelf interior-point methods can
be modified to exploit the structure of clique tree conversion, by
forcing a specific choice of fill-reducing permutation in factorizing
the normal equation. For partially separable semidefinite programs,
the resulting modified solver achieves a \emph{guaranteed} per-iteration
cost of $O(\omega^{6}n)$ time and $O(\omega^{4}n)$ memory on the
dualized version of the clique tree conversion. 

Our complexity guarantees are independent of the actual algorithm
used to factorize the normal equation. Most off-the-shelf interior-point
methods use a standard implementation of the multifrontal method~\cite{duff1983multifrontal,liu1992multifrontal},
but further efficiency can be gained by adopting a parallel and/or
distributed implementation. For example, the interior-point method
of Khoshfetrat Pakazad et al.~\cite{khoshfetrat2017distributed,pakazad2017distributed}
factorizes the normal equation using a message passing algorithm,
which can be understood as a distributed implementation of the multifrontal
method. Of course, distributed algorithms are most efficient when
the workload is evenly distributed, and when communication is minimized.
It remains an important future work to understand these issues in
the context of the sparsity patterns analyzed within this paper.

Finally, a reviewer noted that if the original problem (\ref{eq:SDP})
has a low-rank solution, then the interior-point method iterates approach
a low-dimensional face of the semidefinite cone, which could present
conditioning issues. In contrast, the clique tree conversion might
expect solutions strictly in the interior of the semidefinite cone,
which may be better conditioned. It remains an interesting future
direction to understand the relationship in complementarity, uniqueness,
and conditioning~\cite{alizadeh1997complementarity} between (\ref{eq:SDP})
and its clique tree conversion.

\section{Main results}

\subsection{Assumptions}

To guarantee an exact reformulation, clique tree conversion chooses
the index sets $J_{1},\ldots,J_{\ell}$ in (\ref{eq:relax}) as the
\emph{bags} of a \emph{tree decomposition} for the \emph{sparsity
graph} of the data matrices $C,A_{1},\ldots,A_{m}$. Accordingly,
the parameter $\omega$ in (\ref{eq:omega_def-1}) can only be small
if the sparsity graph has a small \emph{treewidth}. Below, we define
a graph $G$ by its vertex set $V(G)\subseteq\{1,2,\ldots,n\}$ and
its edge set $E(G)\subseteq V(G)\times V(G)$.
\begin{definition}[Sparsity graph]
\label{def:sparsitygraph}The $n\times n$ matrix $M$ (resp. the
set of $n\times n$ matrices $\{M_{1},\ldots,M_{m}\}$) is said to
have \emph{sparsity graph} $G$ if $G$ is an undirected simple graph
on $n$ vertices $V(G)=\{1,\ldots,n\}$ and that $(i,j)\in E(G)$
if $M[i,j]\ne0$ (resp. if there exists $M\in\{M_{1},\ldots,M_{m}\}$
such that $M[i,j]\ne0$). 
\end{definition}
\begin{definition}[Tree decomposition]
\label{def:treedecomp}A \emph{tree decomposition} $\T$ of a graph
$G$ is a pair $(\J,T)$, where each \emph{bag} of vertices $J_{j}\in\J$
is a subset of $V(G)$, and $T$ is a tree on $|\J|\le n$ vertices,
such that:
\begin{enumerate}
\item (Vertex cover) For every $v\in V(G)$, there exists $J_{k}\in\J$
such that $v\in J_{k}$;
\item (Edge cover) For every $(u,v)\in E(G)$, there exists $J_{k}\in\J$
such that $u\in J_{k}$ and $v\in J_{k}$; and
\item (Running intersection) If $v\in J_{i}$ and $v\in J_{j}$, then we
also have $v\in J_{k}$ for every $k$ that lies on the path from
$i$ to $j$ in the tree $T$.
\end{enumerate}
The \emph{width} $\wid(\T)$ of the tree decomposition $\T=(\J,T)$
is the size of its largest bag minus one, as in $\max\{|J_{k}|:J_{k}\in\J\}-1.$
The \emph{treewidth} $\tw(G)$ of the graph $G$ is the minimum width
amongst all tree decompositions $\T$.
\end{definition}
Throughout this paper, we make the implicit assumption that a tree
decomposition with small width is known a priori for the sparsity
graph. In practice, such a tree decomposition can usually be found
using fill-reducing heuristics for sparse linear algebra; see Section~\ref{sec:Preliminaries}.

We also make two explicit assumptions, which are standard in the literature
on interior-point methods.

\begin{assume}[Linear independence]\label{ass:lin}We have $\sum_{i=1}^{m}y_{i}A_{i}=0$
if and only if $y=0$.

\end{assume}

The assumption is without loss of generality, because it can either
be enforced by eliminating $A_{i}\bullet X=b_{i}$ for select $i$,
or else these constraints are not consistent for all $i$. Under Assumption~\ref{ass:lin},
the total number of constraints is bounded $m\le\omega n$ (due to
the fact that $|E(G)|\le n\cdot\tw(G)$~\cite{arnborg1987complexity}).

\begin{assume}[Primal-dual Slater's condition]\label{ass:slater}There
exist $X\succ0,$ $y,$ and $S\succ0$ satisfying $A_{i}\bullet X=b_{i}$
for all $i\in\{1,\ldots,m\}$ and $\sum_{i=1}^{m}y_{i}A_{i}+S=C$. 

\end{assume}

In fact, our proofs solve the homogeneous self-dual embedding~\cite{ye1994homogeneous},
so our conclusions can be extended with few modifications to a much
larger array of problems that mostly do not satisfy Assumption~\ref{ass:slater};
see de~Klerk~et~al.~\cite{de2000self} and Permenter~et~al.~\cite{permenter2017solving}.
Nevertheless, we adopt Assumption~\ref{ass:slater} to simplify our
discussions, by focusing our attention towards the computational aspects
of the interior-point method, and away from the theoretical intricacies
of the self-dual embedding. 

\subsection{Partially separable}

We define the class of \emph{partially separable semidefinite program}
based on the partially separable cones introduced by Sun, Andersen,
and Vandenberghe~\cite{sun2014decomposition}. The general concept
of partial separability is due to Griewank and Toint~\cite{griewank1982partitioned}.
\begin{definition}[Partially separable]
\label{def:decoupled}Let $\T=(\J,T)$ be a tree decomposition for
the sparsity graph of $C,A_{1},\ldots,A_{m}$. The matrix $A_{i}$
is said to be \emph{partially separable }on $\T$ if there exist $J_{j}\in\J$
and some choice of $A_{i,j}$ such that 
\[
A_{i}\bullet X=A_{i,j}\bullet X[J_{j},J_{j}]
\]
for all $n\times n$ matrices $X$. We say that (\ref{eq:SDP}) is
\emph{partially separable} on $\T$ if every constraint matrix $A_{1},\ldots,A_{m}$
is partially separable on $\T$.
\end{definition}
Due to the edge cover property of the tree decomposition, any $A_{i}$
that indexes a single element of $X$ (can be written as $A_{i}\bullet X=X[j,k]$
for suitable $j,k$) is automatically partially separable on any valid
tree decomposition $\T$. In this way, many of the classic semidefinite
relaxations for NP-hard combinatorial optimization problems can be
shown as partially separable.
\begin{example}[MAXCUT and MAX $k$-CUT]
Let $C$ be the (weighted) Laplacian matrix for a graph $G$ with
$n$ vertices. Frieze and Jerrum~\cite{frieze1997improved} proposed
a randomized algorithm to solve MAX $k$-CUT with an approximation
ratio of $1-1/k$ based on solving
\begin{alignat}{3}
 & \text{maximize }\quad & \frac{k-1}{2k}C\bullet X & \tag{MkC}\label{eq:MkC}\\
 & \text{subject to} & X[i,i] & =1\quad & \text{for all } & i\in\{1,\ldots,n\}\nonumber \\
 &  & X[i,j] & \ge\frac{-1}{k-1}\quad & \text{for all } & (i,j)\in E(G)\nonumber \\
 &  & X & \succeq0.\nonumber 
\end{alignat}
The classic Goemans\textendash Williamson 0.878~algorithm~\cite{goemans1995improved}
for MAXCUT is recovered by setting $k=2$ and removing the redundant
constraint $X[i,j]\ge-1$. In both the MAXCUT relaxation and the MAX
$k$-CUT relaxation, observe that each constraint affects a single
matrix element in $X$, so the problem is partially separable on any
tree decomposition. \qed
\end{example}
\begin{example}[Lovasz Theta]
The Lovasz number $\vartheta(G)$ of a graph $G$~\cite{lovasz1979shannon}
is the optimal value to the following dual semidefinite program 
\begin{alignat}{2}
 & \text{minimize }\quad &  & \lambda\tag{LT}\label{eq:LT}\\
 & \text{subject to} &  & \one\one^{T}-\sum_{(i,j)\in E}y_{i,j}(e_{i}e_{j}^{T}+e_{j}e_{i}^{T})\preceq\lambda I\nonumber 
\end{alignat}
over $\lambda\in\R$ and $y_{i,j}\in\R$ for $(i,j)\in E(G)$. Here,
$e_{j}$ is the $j$-th column of the $n\times n$ identity matrix
and $\one$ is the length-$n$ vector-of-ones. Problem (\ref{eq:LT})
is not partially separable. However, given that $\vartheta(G)\ge1$
holds for all graphs $G$, we may divide the linear matrix inequality
through by $\lambda$, redefine $y\gets y/\lambda$, apply the Schur
complement lemma, and take the Lagrangian dual to yield a sparse formulation
\begin{alignat}{2}
 & \text{minimize }\quad & \begin{bmatrix}I & \one\\
\one^{T} & 0
\end{bmatrix}\bullet X & \tag{LT\ensuremath{'}}\label{eq:LTp}\\
 & \text{subject to} & X[i,j] & =0\quad\text{for all }(i,j)\in E\nonumber \\
 &  & X[n+1,n+1] & =1\nonumber \\
 &  & X & \succeq0.\nonumber 
\end{alignat}
Each constraint affects a single matrix element in $X$, so (\ref{eq:LTp})
is again partially separable on any tree decomposition.\qed
\end{example}
We remark that instances of the MAXCUT, MAX $k$-CUT, and Lovasz Theta
problems constitute a significant part of the DIMACS~\cite{pataki2002dimacs}
and the SDPLIB~\cite{borchers1999sdplib} test libraries. In Section~\ref{sec:dctc},
we prove that partially separable semidefinite programs like these
admit a clique tree conversion reformulation that can be dualized
and then solved using an interior-point method in $O(n^{1.5}\log(1/\epsilon))$
time, under the assumption that the parameter $\omega$ in (\ref{eq:omega_def-1})
is significantly smaller than $n$. Moreover, we prove in Section~\ref{sec:splitting}
that this reformulation can be efficiently found using an algorithm
based on the running intersection property of the tree decomposition.
Combining these results with an efficient low-rank matrix completion
algorithm~\cite[Algorithm~2]{sun2015decomposition} yields the following.
\begin{theorem}
\label{thm:comp2}Let $\T=(\{J_{1},\ldots,J_{\ell}\},T)$ be a tree
decomposition for the sparsity graph of $C,A_{1},\ldots,A_{m}$. If
(\ref{eq:SDP}) is partially separable on $\T$, then under Assumptions~\ref{ass:lin}
\&~\ref{ass:slater}, there exists an algorithm that computes $U\in\R^{n\times\omega}$,
$y\in\R^{m}$, and $S\succeq0$ satisfying
\[
\sqrt{\sum_{i=1}^{m}|A_{i}\bullet UU^{T}-b_{i}|^{2}}\le\epsilon,\quad\left\Vert \sum_{i=1}^{m}y_{i}A_{i}+S-C\right\Vert _{F}\le\epsilon,\quad\frac{UU^{T}\bullet S}{n}\le\epsilon
\]
in $O(\omega^{6.5}n^{1.5}\log(1/\epsilon))$ time and $O(\omega^{4}n)$
space, where $\omega=\max_{j}|J_{j}|=1+\wid(\T)$ and $\|M\|_{F}=\sqrt{M\bullet M}$
denotes the Frobenius norm.
\end{theorem}
The proof of Theorem~\ref{thm:comp2} is given at the end of Section~\ref{sec:splitting}.

\subsection{Network flow}

Problems that are not partially separable can be systematically separated
by introducing \emph{auxiliary variables}. The complexity of solving
the resulting problem then becomes parameterized by the number of
additional auxiliary variables. In a class of graph-based semidefinite
programs that we call \emph{network flow semidefinite programs}, the
number of auxiliary variables can be bounded using properties of the
tree decomposition. 
\begin{definition}[Network flow]
\label{def:networkflow}Given a graph $G=(V,E)$ on $n$ vertices
$V=\{1,\ldots,n\}$, we say that the linear constraint $A\bullet X=b$
is a \emph{network flow constraint} (at vertex $k$) if the $n\times n$
constraint matrix $A$ can be rewritten 
\[
A=\alpha_{k}e_{k}e_{k}^{T}+\frac{1}{2}\sum_{(j,k)\in E}\alpha_{j}(e_{j}e_{k}^{T}+e_{k}e_{j}^{T}),
\]
in which $e_{k}$ is the $k$-th column of the identity matrix and
$\{\alpha_{j}\}$ are scalars. We say that an instance of (\ref{eq:SDP})
is a \emph{network flow semidefinite program} if every constraint
matrix $A_{1},\ldots,A_{m}$ is a network flow constraint, and $G$
is the sparsity graph for the objective matrix $C$.
\end{definition}
Such problems frequently arise on physical networks subject to Kirchhoff's
conservation laws, such as electrical circuits and hydraulic networks.
\begin{example}[Optimal power flow]
The AC optimal power flow (ACOPF) problem is a nonlinear, nonconvex
optimization that plays a vital role in the operations of an electric
power system. Let $G$ be a graph representation of the power system.
Then, ACOPF has a well-known semidefinite relaxation 
\begin{alignat}{2}
\text{minimize }\quad &  & \sum_{i\in W}(f_{i,i}X[i,i]+\sum_{(i,j)\in E(G)}\mathrm{Re}\{f_{i,j}X[i,j]\}) & \tag{OPF}\label{eq:OPF}
\end{alignat}
over a Hermitian matrix variable $X$, subject to
\begin{align*}
a_{i,i}X[i,i]+\sum_{(i,j)\in E(G)}\mathrm{Re}\{a_{i,j}X[i,j]\} & \le b_{i}\quad\text{ for all }i\in V(G)\\
\mathrm{Re}\{c_{i,j}X[i,j]\} & \le d_{i,j}\quad\text{ for all }(i,j)\in E(G)\\
X & \succeq0.
\end{align*}
Here, each $a_{i,j}$ and $c_{i,j}$ is a complex vector, each $b_{i}$
and $d_{i,j}$ is a real vector, and $W\subseteq V(G)$ is a subset
of vertices. If a rank-1 solution $X^{\star}$ is found, then the
relaxation (\ref{eq:OPF}) is exact, and a globally-optimal solution
to the original NP-hard problem can be extracted. Clearly, each constraint
in (\ref{eq:OPF}) is a network flow constraint, so the overall problem
is also a network flow semidefinite program. \qed
\end{example}
In Section~\ref{sec:Auxillary-variables}, we prove that network
flow semidefinite programs can be reformulated in closed-form, dualized,
and efficiently solved using an interior-point method.
\begin{theorem}
\label{thm:comp3}Let (\ref{eq:SDP}) be a network flow semidefinite
program on a graph $G$ on $n$ vertices, and let $\T=(\{J_{1},\ldots,J_{\ell}\},T)$
be a tree decomposition for $G$. Then, under Assumptions~\ref{ass:lin}
\&~\ref{ass:slater}, there exists an algorithm that computes $U\in\R^{n\times\omega}$,
$y\in\R^{m}$, and $S\succeq0$ satisfying
\[
\sqrt{\sum_{i=1}^{m}|A_{i}\bullet UU^{T}-b_{i}|^{2}}\le\epsilon,\quad\left\Vert \sum_{i=1}^{m}y_{i}A_{i}+S-C\right\Vert _{F}\le\epsilon,\quad\frac{UU^{T}\bullet S}{n}\le\epsilon
\]
in 
\begin{multline*}
O((\omega+d_{\max}m_{k})^{3.5}\cdot\omega^{3.5}\cdot n^{1.5}\cdot\log(1/\epsilon))\text{ time }\\
\text{and }O((\omega+d_{\max}m_{k})^{2}\cdot\omega^{2}\cdot n)\text{ memory}
\end{multline*}
where: 
\begin{itemize}
\item $\omega=\max_{j}|J_{j}|=1+\wid(\T)$, 
\item $d_{\max}$ is the maximum degree of the tree $T$, 
\item $m_{k}$ is the maximum number of network flow constraints at any
vertex $k\in V(G)$.
\end{itemize}
\end{theorem}
The proof of Theorem~\ref{thm:comp2} is given at the end of Section~\ref{sec:Auxillary-variables}.

\section{\label{sec:Preliminaries}Preliminaries}

\subsection{Notation}

The sets $\R^{n}$ and $\R^{m\times n}$ are the length-$n$ real
vectors and $m\times n$ real matrices. We use ``MATLAB notation''
in concatenating vectors and matrices:

\[
[a,b]=\begin{bmatrix}a & b\end{bmatrix},\qquad[a;b]=\begin{bmatrix}a\\
b
\end{bmatrix},\qquad\diag(a,b)=\begin{bmatrix}a & 0\\
0 & b
\end{bmatrix},
\]
and the following short-hand to construct them:
\[
[x_{i}]_{i=1}^{n}=\begin{bmatrix}x_{1}\\
\vdots\\
x_{n}
\end{bmatrix},\qquad[x_{i,j}]_{i,j=1}^{m,n}=\begin{bmatrix}x_{1,1} & \cdots & x_{1,n}\\
\vdots & \ddots & \vdots\\
x_{m,1} & \cdots & x_{m,n}
\end{bmatrix}.
\]
The notation $X[i,j]$ refers to the element of $X$ in the $i$-th
row and $j$-th column, and $X[I,J]$ refers to the submatrix of $X$
formed from the rows in $I\subseteq\{1,\ldots,m\}$ and columns in
$J\subseteq\{1,\ldots,n\}$. The Frobenius inner product is $X\bullet Y=\tr(X^{T}Y)$,
and the Frobenius norm is $\|X\|_{F}=\sqrt{X\bullet X}$. We use $\nnz(X)$
to denote the number of nonzero elements in $X$. 

The sets $\S^{n}\subseteq\R^{n\times n},$ $\S_{+}^{n}\subset\S^{n},$
and $\S_{++}^{n}\subset\S_{+}^{n}$ are the $n\times n$ real symmetric
matrices, positive semidefinite matrices, and positive definite matrices,
respective. We write $X\succeq Y$ to mean $X-Y\in\S_{+}^{n}$ and
$X\succ Y$ to mean $X-Y\in\S_{++}^{n}$. The (symmetric) vectorization
\[
\vector(X)=[X[1,1];\sqrt{2}X[2,1];\ldots;\sqrt{2}X[m,1];X[2,2],\ldots]
\]
outputs the lower-triangular part of a symmetric matrix as a vector,
with factors of $\sqrt{2}$ added so that $\vector(X)^{T}\vector(Y)=X\bullet Y$. 

A graph $G$ is defined by its vertex set $V(G)\subseteq\{1,2,3,\ldots\}$
and its edge set $E(G)\subseteq V(G)\times V(G)$. The graph $T$
is a \emph{tree} if it is connected and does not contain any cycles;
we refer to its vertices $V(T)$ as its \emph{nodes}. Designating
a special node $r\in V(T)$ as the \emph{root }of the tree allows
us to define the \emph{parent} $p(v)$ of each node $v\ne r$ as the
first node encountered on the path from $v$ to $r$, and $p(r)=r$
for consistency. The set of children is defined $\ch(v)=\{u\in V(T)\backslash v:p(u)=v\}$.
Note that the edges $E(T)$ are fully determined by the parent pointer
$p$ as $\{v,p(v)\}$ for all $v\ne r$.

The set $\S_{G}^{n}\subseteq\S^{n}$ is the set of $n\times n$ real
symmetric matrices with sparsity graph $G$. We denote $P_{G}(X)=\min_{Y\in\S_{G}^{n}}\|X-Y\|_{F}$
as the Euclidean projection of $X\in\S^{n}$ onto $\S_{G}^{n}$.

\subsection{Tree decomposition via the elimination tree}

The standard procedure for solving $Sx=b$ with $S\succ0$ consists
of a \emph{factorization} step, where $S$ is decomposed into the
unique \emph{Cholesky factor} $L$ satisfying
\begin{equation}
LL^{T}=S,\quad L\text{ is lower-triangular},\quad L_{i,i}>0\quad\text{for all }i,\label{eq:chol1}
\end{equation}
and a \emph{substitution }step, where the two triangular systems $Lu=r$
and $L^{T}x=u$ are back-substituted to yield $x$. 

In the case that $S$ is sparse, the location of nonzero elements
in $L$ encodes a tree decomposition for the sparsity graph of $S$
known as the \emph{elimination tree}~\cite{liu1990role}. Specifically,
define the index sets $J_{1},\ldots,J_{n}\subseteq\{1,\ldots,n\}$
as in 
\begin{equation}
J_{j}=\{i\in\{1,\ldots,n\}:L[i,j]\ne0\},\label{eq:Jjdef1}
\end{equation}
and the tree $T$ via the parent pointers
\begin{equation}
p(j)=\begin{cases}
\min_{i}\{i>j:L[i,j]\ne0\} & |J_{j}|>1,\\
j & |J_{j}|=1.
\end{cases}\label{eq:pdef}
\end{equation}
Then, ignoring perfect numerical cancellation, $\T=(\{J_{1},\ldots,J_{n}\},T)$
is a tree decomposition for the sparsity graph of $S$.

Elimination trees with reduced widths can be obtained by reordering
the rows and columns of $S$ using a \emph{fill-reducing} permutation
$\Pi$, because the sparsity graph of $\Pi S\Pi^{T}$ is just the
sparsity graph of $S$ with its vertices reordered. The minimum width
of an elimination tree over all permutations $\Pi$ is precisely the
treewidth of the sparsity graph of $S$; see Bodlaender et al.~\cite{bodlaender1995approximating}
and the references therein. The general problem is well-known to be
NP-complete in general~\cite{arnborg1987complexity}, but polynomial-time
approximation algorithms exist to solve the problem to a logarithmic
factor~\cite{leighton1988approximate,klein1990leighton,bodlaender1995approximating}.
In practice, heuristics like the \emph{minimum degree}~\cite{george1989evolution}
and \emph{nested dissection}~\cite{lipton1979generalized} are considerably
faster while still producing high-quality choices of $\Pi$. 

Note that the sparsity pattern of $L$ is completely determined by
the sparsity pattern of $S$, and not by its numerical value. The
former can be computed from the latter using a \emph{symbolic} Cholesky
factorization algorithm, a standard routine in most sparse linear
algebra libraries, in time linear to the number of nonzeros in $L$;
see \cite[Section~5]{rose1976algorithmic} and \cite[Theorem~5.4.4]{george1981computer},
and also the discussion in \cite{lipton1979generalized}.

\subsection{\label{subsec:prelim_ctc}Clique tree conversion}

Let $\T=(\{J_{1},\ldots,J_{\ell}\},T)$ be a tree decomposition with
small width for the sparsity graph $G$ of the data matrices $C,A_{1},\ldots,A_{m}$.
We define the graph $F\supseteq G$ by taking each index set $J_{j}$
of $\T$ and interconnecting all pairs of vertices $u,v\in J_{j}$,
as in 
\begin{align}
V(F) & =V(G), & E(F) & =\bigcup_{j=1}^{\ell}\{(u,v):u,v\in J_{j}\}.\label{eq:chordal_compl}
\end{align}
The following fundamental result was first established by Grone et
al.~\cite{grone1984positive}. Constructive proofs allow us to recover
all elements in $X\succeq0$ from only the elements in $P_{F}(X)$
using a closed-form formula.
\begin{theorem}[Grone et al.~\cite{grone1984positive}]
\label{thm:prim_sep}Given $Z\in\S_{F}^{n}$, there exists an $X\succeq0$
satisfying $P_{F}(X)=Z$ if and only if $Z[J_{j},J_{j}]\succeq0$
for all $j\in\{1,2,\ldots,\ell\}$.
\end{theorem}
We can use Theorem~\ref{thm:prim_sep} to reformulate (\ref{eq:SDP})
into a reduced-complexity form. The key is to view (\ref{eq:SDP})
as an optimization over $P_{F}(X)$, since 
\[
C\bullet X=\sum_{i,j=1}^{n}C_{i,j}X_{i,j}=\sum_{(i,j)\in F}C_{i,j}X_{i,j}=C\bullet P_{F}(X),
\]
and similarly $A_{i}\bullet X=A_{i}\bullet P_{F}(X)$. Theorem~\ref{thm:prim_sep}
allows us to account for $X\succeq0$ implicitly, by optimizing over
$Z=P_{F}(X)$ in the following
\begin{alignat}{3}
 & \text{minimize }\quad & C\bullet Z\label{eq:CSDP}\\
 & \text{subject to } & A_{i}\bullet Z & =b_{i} & \qquad\text{for all } & i\in\{1,\ldots,m\},\nonumber \\
 &  & Z[J_{j},J_{j}] & \succeq0 & \text{for all } & j\in\{1,\ldots,\ell\}.\nonumber 
\end{alignat}
Next, we split the principal submatrices into distinct matrix variables,
coupled by the need for their overlapping elements to agree. Define
the \emph{overlap} operator $\NN_{i,j}(\cdot)$ to output the overlapping
elements of two principal submatrices given the latter as input:
\[
\NN_{i,j}(X[J_{j},J_{j}])=X[J_{i}\cap J_{j},\;J_{i}\cap J_{j}]=\NN_{j,i}(X[J_{i},J_{i}]).
\]
The running intersection property of the tree decomposition allows
us to enforce this agreement using $\ell-1$ pairwise block comparisons.
\begin{theorem}[Fukuda et al.~\cite{fukuda2001exploiting}]
\label{thm:ctree}Given $X_{1},X_{2},\ldots,X_{\ell}$ for $X_{j}\in\S^{|J_{j}|}$,
there exists $Z$ satisfying $Z[J_{j},J_{j}]=X_{j}$ for all $j\in\{1,2,\ldots,\ell\}$
if and only if $\NN_{i,j}(X_{j})=\NN_{j,i}(X_{i})$ for all $(i,j)\in E(T)$. 
\end{theorem}
Splitting the objective $C$ and constraint matrices $A_{1},\ldots,A_{m}$
into $C_{1},\ldots,C_{\ell}$ and $A_{1,1},\ldots,A_{m,\ell}$ to
satisfy 
\begin{align}
C_{1}\bullet X[J_{1},J_{1}]+C_{2}\bullet X[J_{2},J_{2}]+\cdots+C_{\ell}\bullet X[J_{\ell},J_{\ell}] & =C\bullet X,\label{eq:split_constr}\\
A_{i,1}\bullet X[J_{1},J_{1}]+A_{i,2}\bullet X[J_{2},J_{2}]+\cdots+A_{i,\ell}\bullet X[J_{\ell},J_{\ell}] & =A_{i}\bullet X,\nonumber 
\end{align}
and applying Theorem~\ref{thm:ctree} yields the following
\begin{alignat}{3}
 & \text{minimize }\quad & \sum_{j=1}^{\ell}C_{j}\bullet X_{j} &  &  & \tag{CTC}\label{eq:CTC}\\
 & \text{subject to } & \sum_{j=1}^{\ell}A_{i,j}\bullet X_{j} & =b_{i} & \text{for all } & i\in\{1,\ldots,m\},\nonumber \\
 &  & \NN_{i,j}(X_{j}) & =\NN_{j,i}(X_{i}) & \qquad\text{for all } & (i,j)\in E(T),\nonumber \\
 &  & X_{j} & \succeq0 & \text{for all } & j\in\{1,\ldots,\ell\},\nonumber 
\end{alignat}
which vectorizes into a linear conic program in standard form
\begin{alignat}{4}
 & \text{minimize } & c^{T}x, &  & \qquad\qquad & \text{maximize } & \begin{bmatrix}b\\
0
\end{bmatrix}^{T}y,\label{eq:vecCTC}\\
 & \text{subject to } & \begin{bmatrix}\A\\
\N
\end{bmatrix}x & =\begin{bmatrix}b\\
0
\end{bmatrix}, &  & \text{subject to } & \begin{bmatrix}\A\\
\N
\end{bmatrix}^{T}y+s & =c,\nonumber \\
 &  & x & \in\K, &  &  & s & \in\K_{*}\nonumber 
\end{alignat}
over the Cartesian product of $\ell\le n$ smaller semidefinite cones
\begin{equation}
\K=\K_{*}=\S_{+}^{|J_{1}|}\times\S_{+}^{|J_{2}|}\times\cdots\times\S_{+}^{|J_{\ell}|}.
\end{equation}
Here, $\A=[\vector(A_{i,j})^{T}]_{i,j=1}^{m,\ell}$ and $c=[\vector(C_{j})]_{j=1}^{\ell}$
correspond to (\ref{eq:split_constr}), and the \emph{overlap constraints}
matrix $\N=[\N_{i,j}]_{i,j=1}^{\ell,\ell}$ is implicitly defined
by the relation
\begin{equation}
\N_{i,j}\vector(X_{j})=\begin{cases}
+\vector(\NN_{p(i),i}(X_{i})) & j=i,\\
-\vector(\NN_{i,p(i)}(X_{p(i)})) & j=p(i),\\
0 & \text{otherwise},
\end{cases}\label{eq:Nmat}
\end{equation}
for every non-root node $i$ on $T$. (To avoid all-zero rows in $\N$,
we define $\N_{i,j}\,\vector(X_{j})$ as the empty length-zero vector
$\R^{0}$ if $i$ is the root node.)

The converted problem (\ref{eq:CTC}) inherits the standard regularity
assumptions from (\ref{eq:SDP}). Accordingly, an interior-point method
is well-behaved in solving (\ref{eq:vecCTC}). (Proofs for the following
statements are deferred to Appendix~\ref{sec:linindepslater}.)
\begin{lemma}[Linear independence]
\label{lem:lin_indep}There exists $[u;v]\ne0$ such that $\A^{T}u+\N^{T}v=0$
if and only if there exists $y\ne0$ such that $\sum_{i}y_{i}A_{i}=0$.
\end{lemma}
\begin{lemma}[Primal Slater]
\label{lem:slaterp}There exists $x\in\Int(\K)$ satisfying $\A x=b$
and $\N x=0$ if and only if there exists an $X\succ0$ satisfying
$A_{i}\bullet X=b_{i}$ for all $i\in\{1,\ldots,m\}$.
\end{lemma}
\begin{lemma}[Dual Slater]
\label{lem:slaterd}There exists $u,v$ satisfying $c-\A^{T}u-\N^{T}v\in\Int(\K_{*})$
if and only if there exists $y$ satisfying $C-\sum_{i}y_{i}A_{i}\succ0$.
\end{lemma}
After an $\epsilon$-accurate solution $X_{1}^{\star},\ldots,X_{\ell}^{\star}$
to (\ref{eq:CTC}) is found, we recover, in closed-form, a \emph{positive
semidefinite completion} $X^{\star}\succeq0$ satisfying $X^{\star}[J_{j},J_{j}]=X_{j}^{\star}$,
which in turn serves as an $\epsilon$-accurate solution to (\ref{eq:SDP}).
Of all possible choices of $X^{\star}$, a particularly convenient
one is the \emph{low-rank completion} $X^{\star}=UU^{T}$, in which
$U$ is a dense matrix with $n$ rows and at most $\omega=\max_{j}|J_{j}|$
columns. While the existence of the low-rank completion was known
since Dancis~\cite{dancis1992positive} (see also Laurent and Varvitsiotis~\cite{laurent2014new}
and Madani et al.~\cite{madani2017finding}), Sun~\cite[Algorithm~2]{sun2015decomposition}
gave an explicit algorithm to compute $U^{\star}$ from $X_{1}^{\star},\ldots,X_{\ell}^{\star}$
in $O(\omega^{3}n)$ time and $O(\omega^{2}n)$ memory. The practical
effectiveness of Sun's algorithm was later validated on a large array
of power systems problems by Jiang~\cite{jiang2017minimum}.

\section{\label{sec:Cost-of-an}Cost of an interior-point iteration on (\ref{eq:CTC})}

When the vectorized version (\ref{eq:vecCTC}) of the converted problem
(\ref{eq:CTC}) is solved using an interior-point method, the cost
of each iteration is dominated by the cost of forming and solving
the \emph{normal equation} (also known as the \emph{Schur complement
equation}) 
\begin{equation}
\begin{bmatrix}\A\\
\N
\end{bmatrix}\D_{s}\begin{bmatrix}\A\\
\N
\end{bmatrix}^{T}\Delta y=\begin{bmatrix}\A\D_{s}\A^{T} & \A\D_{s}\N^{T}\\
\N\D_{s}\A^{T} & \N\D_{s}\N^{T}
\end{bmatrix}\begin{bmatrix}\Delta y_{1}\\
\Delta y_{2}
\end{bmatrix}=\begin{bmatrix}r_{1}\\
r_{2}
\end{bmatrix},\label{eq:Hxr}
\end{equation}
where the \emph{scaling matrix} $\D_{s}$ is block-diagonal with fully-dense
blocks 
\begin{equation}
\D_{s}=\diag(\D_{s,1},\ldots,\D_{s,\ell}),\qquad\D_{s,j}\succ0\quad\text{for all }j\in\{1,\ldots,\ell\}.\label{eq:scal}
\end{equation}
Typically, each dense block in $\D_{s}$ is the Hessian of a log-det
penalty, as in $\D_{s,j}=\nabla^{2}[\log\det(X_{j})]$. The submatrix
$\A\D_{s}\A^{T}$ is often sparse~\cite{sun2014decomposition}, with
a sparsity pattern that coincides with the \emph{correlative sparsity}~\cite{kobayashi2008correlative}
of the problem.

Unfortunately, $\N\D_{s}\N^{T}$ can be fully-dense, even when $\A\D_{s}\A^{T}$
is sparse or even diagonal. To explain, observe from (\ref{eq:Nmat})
that the block sparsity pattern of $\N=[\N_{i,j}]_{i,j=1}^{\ell,\ell}$
coincides with the \emph{incidence matrix} of the tree decomposition
tree $T$. Specifically, for every $i$ with parent $p(i)$, the block
$\N_{i,j}$ is nonzero if and only if $j\in\{i,p(i)\}$. As an immediate
corollary, the block sparsity pattern of $\N\D_{s}\N^{T}$ coincides
with the adjacency matrix of the \emph{line graph} of $T$:
\begin{equation}
\sum_{k=1}^{\ell}\N_{i,k}\D_{s,k}\N_{j,k}^{T}\ne0\quad\iff\quad j\in\{i,p(i)\}\text{ or }p(j)\in\{i,p(i)\}.
\end{equation}
The line graph of a tree is not necessarily sparse. If $T$ were the
star graph on $n$ vertices, then its associated line graph $\varL(T)$
would be the complete graph on $n-1$ vertices. Indeed, consider the
following example.
\begin{example}[Star graph]
\label{exa:star}Given $b\in\R^{n}$, embed $\max\{b^{T}y:\|y\|\le1\}$
into the order-$(n+1)$ semidefinite program: 
\begin{alignat*}{3}
 & \text{minimize } & \tr(X)\\
 & \text{subject to } & X[i,(n+1)] & =b_{i} & \quad\text{for all }i\in\{1,\ldots,n\}\\
 &  & X & \succeq0
\end{alignat*}
The associated sparsity graph $G$ is the star graph on $n+1$ nodes,
and its elimination tree $\T=(\{J_{1},\ldots,J_{n}\},T)$ has index
sets $J_{j}=\{j,n+1\}$ and parent pointer $p(j)=n$. Applying clique
tree conversion and vectorizing yields an instance of (\ref{eq:vecCTC})
with
\begin{align*}
\A & =\begin{bmatrix}e_{2}^{T} &  & 0\\
 & \ddots\\
0 &  & e_{2}^{T}
\end{bmatrix}, & \N & =\begin{bmatrix}e_{3}^{T} &  & 0 & -e_{3}^{T}\\
 & \ddots &  & \vdots\\
0 &  & e_{3}^{T} & -e_{3}^{T}
\end{bmatrix},
\end{align*}
where $e_{j}$ is the $j$-th column of the $3\times3$ identity matrix.
It is straightforward to verify that $\A\D_{s}\A^{T}$ is $n\times n$
diagonal but $\N\D_{s}\N^{T}$ is $(n-1)\times(n-1)$ fully dense
for the $\D_{s}$ in (\ref{eq:scal}). The cost of solving the corresponding
normal equation (\ref{eq:Hxr}) must include the cost of factoring
this fully dense submatrix, which is at least $(n-1)^{3}/3$ operations
and $(n-1)^{2}/2$ units of memory. \qed
\end{example}
On the other hand, observe that the block sparsity graph of $\N^{T}\N$
coincides with the tree graph $T$
\begin{equation}
\sum_{k=1}^{\ell}\N_{k,i}^{T}\N_{k,j}\ne0\quad\iff\quad i=j\text{ or }(i,j)\in E(T).
\end{equation}
Such a matrix is guaranteed to be \emph{block} sparse: sparse over
dense blocks. More importantly, after a \emph{topological} block permutation
$\Pi$, the matrix $\Pi(\N^{T}\N)\Pi^{T}$ factors into $\L\L^{T}$
with \emph{no block fill}.
\begin{definition}[Topological ordering]
\label{def:topological}An ordering\emph{ }$\pi:\{1,2,\ldots,n\}\to V(T)$
on the tree graph $T$ with $n$ nodes is said to be \emph{topological}~\cite[p.~10]{vandenberghe2015chordal}
if, by designating $\pi(n)$ as the root of $T$, each node is indexed
before its parent:
\[
\pi^{-1}(v)<\pi^{-1}(p(v))\qquad\text{ for all }v\ne r,
\]
where $\pi^{-1}(v)$ denotes the index associated with the node $v$.
\end{definition}
\begin{lemma}[No block fill]
\label{lem:blocktree}Let $J_{1},\ldots,J_{n}$ satisfy $\bigcup_{j=1}^{n}J_{j}=\{1,\ldots,d\}$
and $J_{i}\cap J_{j}=\emptyset$ for all $i\ne j$, and let $H\succ0$
be a $d\times d$ matrix satisfying
\[
H[J_{i},J_{j}]\ne0\qquad\implies\qquad(i,j)\in E(T)
\]
for a tree graph $T$ on $n$ nodes. If $\pi$ is a topological ordering
on $T$ and $\Pi$ is a permutation matrix satisfying 
\[
(\Pi H\Pi)[J_{i},J_{j}]=H[J_{\pi(i)},J_{\pi(j)}]\qquad\text{for all }i,j\in\{1,\ldots,n\},
\]
then $\Pi H\Pi^{T}$ factors into $LL^{T}$ where the Cholesky factor
$L$ satisfies
\[
L[J_{i},J_{j}]\ne0\qquad\implies\qquad(\Pi H\Pi)[J_{i},J_{j}]\ne0\qquad\text{ for all }i>j.
\]
Therefore, sparse Cholesky factorization solves $Hx=b$ for $x$ by:
(i) factoring $\Pi H\Pi^{T}$ into $LL^{T}$ in $O(\beta^{3}n)$ operations
and $O(\beta^{2}n)$ memory where $\beta=\max_{j}|J_{j}|$, and (ii)
solving $Ly=\Pi b$ and $L^{T}z=y$ and $x=\Pi^{T}z$ in $O(\beta^{2}n)$
operations and memory. 
\end{lemma}
This is a simple block-wise extension of the tree elimination result
originally due to Parter~\cite{parter1961use}; see also George and
Liu~\cite[Lemma 6.3.1]{george1981computer}. In practice, a topological
ordering can be found by assigning indices $n,n-1,n-2,\ldots$ in
decreasing ordering during a depth-first search traversal of the tree.
In fact, the minimum degree heuristic is guaranteed to generate a
topological ordering~\cite{george1989evolution}.

One way of exploiting the favorable block sparsity of $\N^{T}\N$
is to view the normal equation (\ref{eq:Hxr}) as the Schur complement
equation to an augmented system with $\epsilon=0$:
\begin{equation}
\begin{bmatrix}\D_{s}^{-1} & \A^{T} & \N^{T}\\
\A & -\epsilon I & 0\\
\N & 0 & -\epsilon I
\end{bmatrix}\begin{bmatrix}\Delta x\\
\Delta y_{1}\\
\Delta y_{2}
\end{bmatrix}=\begin{bmatrix}0\\
-r_{1}\\
-r_{2}
\end{bmatrix}.\label{eq:nrmaug}
\end{equation}
Instead, we can solve the \emph{dual} Schur complement equation for
$\epsilon>0$
\begin{equation}
\left(\D_{s}^{-1}+\frac{1}{\epsilon}\A^{T}\A+\frac{1}{\epsilon}\N^{T}\N\right)\Delta x=-\frac{1}{\epsilon}\A^{T}r_{1}-\frac{1}{\epsilon}\A^{T}r_{2}\label{eq:dualschur}
\end{equation}
and recover an approximate solution. Under suitable sparsity assumptions
on $\A^{T}\A$, the block sparsity graph of the matrix in (\ref{eq:dualschur})
coincides with that of $\N^{T}\N$, which is itself the tree graph
$T$. Using sparse Cholesky factorization with a topological block
permutation, we solve (\ref{eq:dualschur}) in linear time and back
substitute to obtain a solution to (\ref{eq:nrmaug}) in linear time.
In principle, a sufficiently small $\epsilon>0$ will approximate
the exact case at $\epsilon=0$ to arbitrary accuracy, and this is
all we need for the outer interior-point method to converge in polynomial
time.

A more subtle way to exploit the block sparsity of $\N^{T}\N$ is
to reformulate (\ref{eq:CTC}) into a form whose normal equation is
exactly (\ref{eq:dualschur}). As we show in the next section, this
is achieved by a simple technique known as dualization.

\section{\label{sec:dctc}Dualized clique tree conversion}

The dualization technique of Löfberg~\cite{lofberg2009dualize} swaps
the roles played by the primal and the dual problems in a linear conic
program, by rewriting a primal standard form problem into dual standard
form, and vice versa. Applying dualization to (\ref{eq:vecCTC}) yields
the following
\begin{alignat}{4}
 & \text{minimize } & \begin{bmatrix}b\\
0
\end{bmatrix}^{T}x_{1} &  & \qquad & \text{maximize } & -c^{T}y\label{eq:dualCTC}\\
 & \text{subject to } & \begin{bmatrix}\A\\
\N
\end{bmatrix}^{T}x_{1}-x_{2} & =-c, &  & \text{subject to } & \begin{bmatrix}\A\\
\N
\end{bmatrix}y+s_{1} & =\begin{bmatrix}b\\
0
\end{bmatrix},\nonumber \\
 &  & x_{1}\in\R^{f},\,x_{2} & \in\K. &  &  & -y+s_{2} & =0,\nonumber \\
 &  &  &  &  &  & s_{1}\in\{0\}^{f},\,s_{2} & \in\K.\nonumber 
\end{alignat}
where we use $f$ to denote the number of equality constraints in
(\ref{eq:CTC}). Observe that the dual problem in (\ref{eq:dualCTC})
is identical to the primal problem in (\ref{eq:vecCTC}), so that
a dual solution $y^{\star}$ to (\ref{eq:dualCTC}) immediately serves
as a primal solution to (\ref{eq:vecCTC}), and hence also (\ref{eq:CTC}).

Modern interior-point methods solve (\ref{eq:dualCTC}) by embeding
the free variable $x_{1}\in\R^{f}$ and fixed variable $s_{1}\in\{0\}^{f}$
into a second-order cone (see Sturm~\cite{sturm1999using} and Andersen~\cite{andersen2002handling}):
\begin{alignat}{4}
 & \text{minimize } & \begin{bmatrix}b\\
0
\end{bmatrix}^{T}x_{1} &  & \qquad & \text{maximize } & -c^{T}y\label{eq:dualCTC-2}\\
 & \text{subject to } & \begin{bmatrix}\A\\
\N
\end{bmatrix}^{T}x_{1}-x_{2} & =-c, &  & \text{subject to } & \begin{bmatrix}\A\\
\N
\end{bmatrix}y+s_{1} & =\begin{bmatrix}b\\
0
\end{bmatrix},\nonumber \\
 &  & \|x_{1}\|\le x_{0},\,x_{2} & \in\K. &  &  & -y+s_{2} & =0,\nonumber \\
 &  &  &  &  &  & s_{0} & =0,\nonumber \\
 &  &  &  &  &  & \|s_{1}\|\le s_{0},\,s_{2} & \in\K.\nonumber 
\end{alignat}
When (\ref{eq:dualCTC-2}) is solved using an interior-point method,
the normal equation solved at each iteration takes the form
\begin{equation}
\left(\D_{s}+\begin{bmatrix}\A\\
\N
\end{bmatrix}^{T}\D_{f}\begin{bmatrix}\A\\
\N
\end{bmatrix}\right)\Delta y=r\label{eq:Hxr_dual}
\end{equation}
where $\D_{s}$ is comparable as before in (\ref{eq:scal}), and 
\begin{equation}
\D_{f}=\sigma I+ww^{T},\qquad\sigma>0\label{eq:scal_f}
\end{equation}
is the rank-1 perturbation of a scaled identity matrix. The standard
procedure, as implemented in SeDuMi~\cite{sturm1999using,sturm2002implementation}
and MOSEK~\cite{andersen2003implementing}, is to form the sparse
matrix $\H$ and dense vector $\q$, defined
\begin{align}
\H & =\D_{s}+\sigma\A^{T}\A+\sigma\N^{T}\N, & \q & =\begin{bmatrix}\A^{T} & \N^{T}\end{bmatrix}w.\label{eq:Hqdef}
\end{align}
and then solve (\ref{eq:Hxr_dual}) using a rank-1 update\footnote{To keep our derivations simple, we perform the rank-1 update using
the Sherman\textendash Morrison\textemdash Woodbury (SMW) formula.
In practice, the product-form Cholesky Factor (PFCF) formula of Goldfarb
and Scheinberg~\cite{goldfarb2005product} is more numerically stable
and more widely used~\cite{sturm1999using,sturm2002implementation}.
Our complexity results remain valid in either cases because the PFCF
is a constant factor of approximately two times slower than the SWM~\cite{goldfarb2005product}.}
\begin{equation}
\Delta y=(\H+\q\q^{T})^{-1}r=\left(I-\frac{(\H^{-1}\q)\q^{T}}{1+\q^{T}(\H^{-1}\q)}\right)\H^{-1}r,\label{eq:update}
\end{equation}
at a cost comparable to the solution of $\H u=r$ for two right-hand
sides. (In Appendix~\ref{sec:inequality}, we repeat these derivations
for the version of (\ref{eq:vecCTC}) in which $\A x=b$ is replaced
by the inequality $\A x\le b$.)

The matrix $\H$ is \emph{exactly} the dual Schur complement derived
in (\ref{eq:dualschur}) with $\sigma=1/\epsilon$. If the $\A^{T}\A$
shares its block sparsity pattern with $\N^{T}\N$, then the block
sparsity graph of $\H$ coincides with the tree graph $T$, and $\H u=r$
can be solved in linear time. The cost of making the rank-1 update
is also linear time, so the cost of solving the normal equation is
linear time. 
\begin{lemma}[Linear-time normal equation]
\label{lem:normal}Let there exist $v_{i}\in V(T)$ for each $i\in\{1,\ldots,m\}$
such that
\begin{equation}
A_{i,j}\ne0\qquad\implies\qquad j=v_{i}\text{ or }j=p(v_{i}).\label{eq:Asparsity}
\end{equation}
Define $\H$ and $\q$ according to (\ref{eq:Hqdef}). Then, under
Assumption~\ref{ass:lin}:
\begin{enumerate}
\item (Forming) It costs $O(\omega^{6}n)$ time and $O(\omega^{4}n)$ space
to form $\H$ and $\q$, where $\omega=\max_{j}|J_{j}|=1+\wid(\T)$.
\item (Factoring) Let $\pi$ be a topological ordering on $T$, and define
the associated block topological permutation $\Pi$ as in Lemma~\ref{lem:blocktree}.
Then, it costs $O(\omega^{6}n)$ time and $O(\omega^{4}n)$ space
to factor $\Pi\H\Pi^{T}$ into $\L\L^{T}$. 
\item (Solving) Given $r$, $\q$, $\Pi$, and the Cholesky factor $\L$
satisfying $\L\L^{T}=\Pi\H\Pi^{T}$, it costs $O(\omega^{4}n)$ time
and space to solve $(\H+\q\q^{T})u=r$ for $u$.
\end{enumerate}
\end{lemma}
\begin{proof}
For an instance of (\ref{eq:CTC}), denote $\ell=|\mathcal{J}|\le n$
as its number of conic constraints, and $d=\frac{1}{2}\sum_{j=1}^{\ell}|J_{j}|(|J_{j}|+1)\le\omega^{2}\ell$
as its total number of variables. Under linear independence (Assumption~\ref{ass:lin}),
the constraint matrix $[\A;\N]$ associated with (\ref{eq:CTC}) has
$d$ columns and at most $d$ rows (Lemma~\ref{lem:lin_indep}).
Write $\xi_{i}^{T}$ as the $i$-th row of $[\A;\N]$, and assume
without loss of generality that $[\A;\N]$ has exactly $d$ rows.
Observe that $\nnz(\xi_{i})\le\omega(\omega+1)$ by the definition
of $\N$ (\ref{eq:Nmat}) and the hypothesis on $\A$ via (\ref{eq:Asparsity}),
so $\nnz([\A;\N])\le2\omega^{4}\ell$. 

\textbf{(i)} We form $\H$ by setting $\H\gets\D_{s}$ and then adding
$\H\gets\H+\sigma\xi_{i}\xi_{i}^{T}$ one at a time, for $i\in\{1,2,\ldots,d\}$.
The first step forms $\D_{s}=\diag(\D_{s}^{(1)},\D_{s}^{(2)},\ldots,\D_{s}^{(\ell)})$
where $\D_{s}^{(j)}=W_{j}\otimes W_{j}$. Each $\nnz(\D_{s}^{(j)})=\nnz(W_{j})^{2}=|J_{j}|^{2}(|J_{j}|+1)^{2}/4\le\omega^{4}$
for $j\in\{1,2,\ldots,\ell\}$, so the total cost is $O(\omega^{4}n)$
time and space. The second step adds $\nnz(\xi_{i}\xi_{i}^{T})^{2}\le\omega^{2}(\omega+1)^{2}$
nonzeros per constraint over $d$ total constraints, for a total cost
of $O(\omega^{6}n)$ time and apparently $O(\omega^{6}n)$ space.
However, by the definition of $\N$ (\ref{eq:Nmat}) and the hypothesis
on $\A$ via (\ref{eq:Asparsity}), the $(j,k)$-th off-diagonal block
of $\xi_{i}\xi_{i}^{T}$ is nonzero only if $(j,k)$ is an edge of
the tree $T$, as in
\[
\xi_{i}\xi_{i}^{T}[J_{j},J_{k}]\ne0\qquad\implies\qquad(j,k)\in E(T)\qquad\text{for all }j\ne k.
\]
Hence, adding $\mathbf{H}\gets\mathbf{H}+\sigma\xi_{i}\xi_{i}^{T}$
one at a time results in at most $|V(T)|+|E(T)|$ dense blocks of
at most $\frac{1}{2}\omega(\omega+1)\times\frac{1}{2}\omega(\omega+1)$,
for a total memory cost of $O(\omega^{4}n)$. 

\textbf{(ii)} We form $\q=[\A^{T},\N^{T}]w_{1}$ using a sparse matrix-vector
product. Given that $\nnz(w_{1})\le d$ and $\nnz([\A;\N])\le2\omega^{4}\ell$,
this step costs $O(\omega^{4}n)$ time and space.

\textbf{(iii)} We partition $\H$ into $[\H_{i,j}]_{i,j=1}^{\ell}$
to reveal a block sparsity pattern that coincides with the adjacency
matrix of $T$:
\begin{align*}
\H_{i,j} & =\begin{cases}
\D_{s,i}+\sigma\sum_{k=1}^{\ell}\N_{k,i}^{T}\N_{k,i}+\sigma\sum_{q=1}^{m}a_{q,i}a_{q,i}^{T} & i=j\\
\sigma\sum_{k=1}^{\ell}\N_{k,i}^{T}\N_{k,j}+\sigma\sum_{q=1}^{m}a_{q,i}a_{q,j}^{T} & (i,j)\in E(T)\\
0 & \text{otherwise}
\end{cases}
\end{align*}
where $a_{q,i}=\vector(A_{q,i})$. According to Lemma~\ref{lem:blocktree},
the permuted matrix $\Pi\H\Pi^{T}$ factors into $\L\L^{T}$ with
no block fill in $O(\omega^{6}n)$ time and $O(\omega^{4}n)$ space,
because each block $\H_{i,j}$ for $i,j\in\{1,2,\ldots,\ell\}$ is
at most order $\frac{1}{2}\omega(\omega+1)$.

\textbf{(iv)} Using the rank-1 update formula (\ref{eq:update}),
the cost of solving $(\H+\q\q^{T})u=r$ is the same as the cost of
solving $\H u=r$ for two right-hand sides, plus algebraic manipulations
in $O(d)=O(\omega^{2}n)$ time. Applying Lemma~\ref{lem:blocktree}
shows that the cost of solving $\H u=r$ for each right-hand side
is $O(\omega^{4}n)$ time and space. \qed
\end{proof}

Incorporating the block topological permutation of Lemma~\ref{lem:normal}
within any off-the-self interior-point method yields a fast interior-point
method with near-linear time complexity.
\begin{theorem}[Near-linear time]
\label{thm:lintime}Let $\T=(\{J_{1},\ldots,J_{\ell}\},T)$ be a
tree decomposition for the sparsity graph of $C,A_{1},\ldots,A_{m}\in\S^{n}$.
In the corresponding instance of (\ref{eq:CTC}), let each constraint
be written
\begin{equation}
\sum_{j=1}^{\ell}A_{i,j}\bullet X_{j}=A_{i,j}\bullet X_{j}+A_{i,k}\bullet X_{k}=b_{i}\qquad(j,k)\in E(T).\label{eq:structA}
\end{equation}
Under Assumptions~\ref{ass:lin} \&~\ref{ass:slater}, there exists
an algorithm that computes an iterate $(x,y,s)\in\K\times\R^{p}\times\K_{*}$
satisfying
\begin{align}
\left\Vert \begin{bmatrix}\A\\
\N
\end{bmatrix}x-\begin{bmatrix}b\\
0
\end{bmatrix}\right\Vert  & \le\epsilon, & \left\Vert \begin{bmatrix}\A\\
\N
\end{bmatrix}^{T}y-s+c\right\Vert  & \le\epsilon, & \frac{x^{T}s}{\sum_{j=1}^{\ell}|J_{j}|} & \le\epsilon\label{eq:eps_accurate}
\end{align}
in $O(\omega^{6.5}n^{1.5}\log(1/\epsilon))$ time and $O(\omega^{4}n)$
space, where $\omega=\max_{j}|J_{j}|=1+\wid(\T)$.
\end{theorem}
For completeness, we give a proof of Theorem~\ref{thm:lintime} in
Appendix~\ref{subsec:complexity}, based on the primal-dual interior-point
method found in SeDuMi~\cite{sturm1999using,sturm2002implementation}.
Our proof amounts to replacing the fill-reducing permutation\textemdash usually
a minimum degree ordering\textemdash by the block topological permutation
of of Lemma~\ref{lem:normal}. In practice, the minimum degree ordering
is often approximately block topological, and as such, Theorem~\ref{thm:lintime}
is often attained by off-the-shelf implementations without modification.

\begin{algorithm}
\caption{\label{alg:dctc}Dualized clique tree conversion}

\textbf{Input.} Data vector $b\in\R^{m}$, data matrices $C,A_{1},\ldots,A_{m}$,
and tree decomposition $\T=(\{J_{1},\ldots,J_{\ell}\},T)$ for the
sparsity graph of the data matrices.

\textbf{Output.} An $\epsilon$-accurate solution of (\ref{eq:SDP})
in factored form $X^{\star}=UU^{T}$, where $U\in\R^{n\times\omega}$
and $\omega=\max_{j}|J_{j}|$.

\textbf{Algorithm.}
\begin{enumerate}
\item (Conversion) Reformulate (\ref{eq:SDP}) into (\ref{eq:CTC}).
\item (Dualization) Vectorize (\ref{eq:CTC}) into (\ref{eq:vecCTC}) and
dualize into (\ref{eq:dualCTC-2}).
\item (Solution) Solve (\ref{eq:dualCTC-2}) as an order-$\nu$ conic linear
program in standard form, using an interior-point method with $O(\sqrt{\nu}\log(1/\epsilon))$
iteration complexity. At each iteration of the interior-point method,
solve the normal equation using sparse Cholesky factorization and
the fill-reducing permutation $\Pi$ in Lemma~\ref{lem:normal}.
Obtain $\epsilon$-accurate solutions $X_{1}^{\star},\ldots,X_{\ell}^{\star}$.
\item (Recovery) Recover $X^{\star}=UU^{T}$ satisfying $X^{\star}[J_{i},J_{i}]=X_{i}^{\star}$
using the low-rank matrix completion algorithm~\cite[Algorithm~2]{sun2015decomposition}.
\end{enumerate}
\end{algorithm}
The complete end-to-end procedure for solving (\ref{eq:SDP}) using
dualized clique tree conversion is summarized as Algorithm~\ref{alg:dctc}.
Before we can use Algorithm~\ref{alg:dctc} to prove our main results,
however, we must first address the cost of the pre-processing involved
in Step~1. Indeed, naively converting (\ref{eq:SDP}) into (\ref{eq:CTC})
by comparing each nonzero element of $A_{i}$ against each index set
$J_{j}$ would result in $\ell m=O(n^{2})$ comparisons, and this
would cause Step~1 to become the overall bottleneck of the algorithm. 

In the next section, we show that if (\ref{eq:SDP}) is partially
separated, then the cost of Step~1 is no more than $O(\omega^{3}n)$
time and memory. This is the final piece in the proof of Theorem~\ref{thm:comp2}.

\section{\label{sec:splitting}Optimal constraint splitting}

A key step in clique tree conversion is the splitting of a given $M\in\S_{F}^{n}$
into $M_{1},\ldots,M_{\ell}$ that satisfy
\begin{equation}
M_{1}\bullet X[J_{1},J_{1}]+M_{2}\bullet X[J_{2},J_{2}]+\cdots+M_{\ell}\bullet X[J_{\ell},J_{\ell}]=M\bullet X\qquad\text{for all }X\in\S^{n}.\label{eq:split_goal}
\end{equation}
The choice is not unique, but has a significant impact on the complexity
of an interior-point solution. The problem of choosing the \emph{sparsest}
choice with the fewest nonzero $M_{j}$ matrices can be written 
\begin{equation}
\Set^{\star}=\underset{\Set\subseteq\{1,\ldots,\ell\}}{\text{ minimize }}\quad|\Set|\quad\text{ subject to }\quad\bigcup_{j\in\Set}(J_{j}\times J_{j})\supseteq\mathcal{M},\label{eq:set_cover}
\end{equation}
where $\mathcal{M}=\{(i,j):M[i,j]\ne0\}$ are the nonzero matrix elements
to be covered. Problem (\ref{eq:set_cover}) is an instance of SET
COVER, one of Karp's 21 NP-complete problems, but becomes solvable
in polynomial time given a tree decomposition (with small width) for
the covering sets~\cite{guo2006exact}. 

In this section, we describe an algorithm that computes the sparsest
splitting for each $M$ in $O(\nnz(M))$ time and space, after a precomputation
set taking $O(\omega n)$ time and memory. Using this algorithm, we
convert a partially separable instance of (\ref{eq:SDP}) into (\ref{eq:CTC})
in $O(\omega^{3}n)$ time and memory. Then, give a complete proof
to Theorem~\ref{thm:comp2} by using this algorithm to convert (\ref{eq:SDP})
into (\ref{eq:CTC}) in Step~1 of Algorithm~\ref{alg:dctc}.

Our algorithm is adapted from the leaf-pruning algorithm of Guo and
Niedermeier~\cite{guo2006exact}, but appears to be new within the
context of clique tree conversion. Observe that the covering sets
inherit the \emph{edge cover} and \emph{running intersection }properties
of $\T$:
\begin{align}
\bigcup_{j=1}^{\ell}(J_{j}\times J_{j}) & \supseteq\mathcal{M}\quad\text{for all possible choices of }\mathcal{M},\label{eq:tree_like1}\\
(J_{i}\times J_{i})\cap(J_{j}\times J_{j}) & \subseteq(J_{k}\times J_{k})\quad\text{for all }k\text{ on the path from }i\text{ to }j.\label{eq:tree_like2}
\end{align}
For every leaf node $j$ with parent node $p(j)$ on $T$, property
(\ref{eq:tree_like2}) implies that the subset $(J_{j}\times J_{j})\backslash(J_{p(j)}\times J_{p(j)})$
contains elements \emph{unique} to $J_{j}\times J_{j}$, because $p(j)$
lies on the path from $j$ to all other nodes in $T$. If $\mathcal{M}$
contains an element from this subset, then $j$ must be included in
the cover set $\mathcal{S}$, so we set $\mathcal{S}\gets\mathcal{S}\cup\{j\}$
and $\mathcal{M}\gets\mathcal{M}\backslash(J_{j}\times J_{j})$; otherwise,
we do nothing. Pruning the leaf node reveals new leaf nodes, and we
repeat this process until the tree $T$ is exhausted of nodes. Then,
property (\ref{eq:tree_like1}) guarantees that $\mathcal{M}$ will
eventually be covered.

\begin{algorithm}[t]
\caption{\label{alg:setcover-tree}Optimal algorithm for splitting constraints}

\textbf{Input.} Data matrices $M_{1},\ldots,M_{m}$. Tree decomposition
$\T=(\{J_{1},\ldots,J_{\ell}\},T)$ for the sparsity graph of the
data matrices.

\textbf{Output.} Split matrices $M_{i,j}$ satisfying $M_{i}\bullet X=\sum_{j=1}^{\ell}M_{i,j}\bullet X[J_{j},J_{j}]$
for all $i\in\{1,\ldots,m\}$ in which the number of nonzero split
matrices $\{j:M_{i,j}\ne0\}$ is minimized.

\textbf{Algorithm.}
\begin{enumerate}
\item (Precomputation) Arbitrarily root $T$, and iterate over $j\in\{1,\ldots,\ell\}$
in any order. For each $j$ with parent $k$, define $U_{j}\equiv J_{j}\backslash J_{k}$.
For the root $j$, define $U_{j}=J_{j}$. For each $k\in U_{j}$,
set $u(k)=j$.
\item (Overestimation) Iterate over $i\in\{1,\ldots,m\}$ in any order.
For each $i$, compute the overestimate $\mathcal{S}_{i}=\bigcup_{(i,j)\in\mathcal{M}}u(j)$
where $\mathcal{M}=\{(j,k):M_{i}[j,k]\ne0\}$ are the nonzeros to
be covered.
\item (Leaf pruning on the overestimation) Iterate over $j\in\Set_{i}$
in topological order on $T$ (children before parents). If $M_{i}[J_{j},U_{j}]\ne0$
then add $j$ to the set cover, and remove the covered elements
\[
M_{i,j}\gets M_{i}[J_{j},J_{j}],\qquad M_{i}[J_{j},J_{j}]\gets0.
\]
If $M_{i}=0$, break. Return to Step 2 for a new value of $i$.
\end{enumerate}
\end{algorithm}
Algorithm~\ref{alg:setcover-tree} is an adaptation of the leaf-pruning
algorithm described above, with three important simplifications. First,
it uses a topological traversal (Definition~\ref{def:topological})
to simulate the process of leaf pruning without explicitly deleting
nodes from the tree. Second, it notes that the unique subset $(J_{j}\times J_{j})\backslash(J_{p(j)}\times J_{p(j)})$
can be written in terms of another unique set $U_{j}$:
\[
(J_{j}\times J_{j})\backslash(J_{p(j)}\times J_{p(j)})=(U_{j}\times J_{j})\cup(J_{j}\times U_{j})\qquad\text{where }U_{j}\equiv J_{j}\backslash J_{p(j)}.
\]
Third, it notes that the unique set $U_{j}$ defined above is a partitioning
of $\{1,\ldots,n\}$, and has a well-defined inverse map. The following
is taken from~\cite{lewis1989fast,pothen1990compact}, where $U_{j}$
is denoted $\mathrm{new}(J_{j})$ and referred to as the ``new set''
of $J_{j}$; see also~\cite{andersen2013logarithmic}.
\begin{lemma}[Unique partition]
Define $U_{j}=J_{j}\backslash J_{p(j)}$ for all nodes $j$ with
parent $p(j)$, and $U_{r}=J_{r}$ for the root node $r$. Then: (i)
$\bigcup_{j=1}^{\ell}U_{j}=\{1,\ldots,n\}$; and (ii) $U_{i}\cap U_{j}=\emptyset$
for all $i\ne j$.
\end{lemma}
In the case that $\mathcal{M}$ contains just $O(1)$ items to be
covered, we may use the inverse map associated with $U_{j}$ to directly
identify covering sets whose unique sets contain elements from $\mathcal{M}$,
without exhaustively iterating through all $O(n)$ covering sets.
This final simplification reduces the cost of processing each $M_{i}$
from linear $O(n)$ time to $O(\nnz(M_{i}))$ time, after setting
up the inverse map in $O(\omega n)$ time and space.
\begin{theorem}
\label{thm:preprocess}Algorithm~\ref{alg:setcover-tree} has complexity
\[
O(\omega n+\nnz(M_{1})+\nnz(M_{2})+\cdots+\nnz(M_{m}))\text{ time and memory,}
\]
where $\omega\equiv1+\wid(\T)$.
\end{theorem}
For partially separable instances of (\ref{eq:SDP}), the sparsest
instance of (\ref{eq:CTC}) contains exactly one nonzero split matrix
$A_{i,j}\ne0$ for each $i$, and Algorithm~\ref{alg:setcover-tree}
is guaranteed to find it. Using Algorithm~\ref{alg:setcover-tree}
to convert (\ref{eq:SDP}) into (\ref{eq:CTC}) in Step~1 of Algorithm~\ref{alg:dctc}
yields the complexity figures quoted in Theorem~\ref{thm:comp2}. 

\begin{proof}[Theorem~\ref{thm:comp2}]By hypothesis, $\T=\{J_{1},\ldots,J_{\ell}\}$
is a tree decomposition for the sparsity graph of the data matrices
$C,A_{1},\ldots,A_{m}$, and (\ref{eq:SDP}) is partially separable
on $\T$. We proceed to solve (\ref{eq:SDP}) using Algorithm~\ref{alg:dctc},
while performing the splitting into $C_{j}$ and $A_{i,j}$ using
Algorithm~\ref{alg:setcover-tree}. Below, we show that each step
of the algorithm costs no more than $O(\omega^{6.5}n^{1.5}\log(1/\epsilon))$
time and $O(\omega^{4}n)$ memory:

\textbf{Step 1 }(Matrix $\A$ and vector $\c$)\textbf{.} We have
$\dim(\S_{G}^{n})=|V(G)|+|E(G)|\le n+n\cdot\wid(\T)\le\omega n$,
and hence $\nnz(C)\le\omega n$. Under partial separability (Definition~\ref{def:decoupled}),
we also have $\nnz(A_{i})\le\omega^{2}$. Assuming linear independence
(Assumption~\ref{ass:lin}) yields $m\le\dim(\S_{G}^{n})\le\omega n$,
and this implies that $\nnz(C)+\sum_{i}\nnz(A_{i})=O(\omega^{3}n)$,
so the cost of forming $\A$ and $\c$ using Algorithm~\ref{alg:dctc}
is $O(\omega^{3}n)$ time and memory via Theorem~\ref{thm:preprocess}.

\textbf{Step 1} (Matrix $\N$)\textbf{.} For $\N=[\N_{i,j}]_{i,j=1}^{\ell}$,
we note that each block $\N_{i,j}$ is diagonal, and hence $\nnz(\N_{i,j})\le\omega^{2}$.
The overall $\N$ contains $\ell$ block-rows, with $2$ nonzero blocks
per block-row, for a total of $2\ell$ nonzero blocks. Therefore,
the cost of forming $\N$ is $\nnz(\N)=O(\omega^{2}n)$ time and memory. 

\textbf{Step 2.} We dualize by forming the matrix $\M=[0,-\A^{T},\N^{T},+I]$
and vectors $\c^{T}=[0,b^{T},0,0]$ and vectors $\b=-c$ in $O(\nnz(\A)+\nnz(\N))=O(\omega^{3}n)$
time and memory.

\textbf{Step 3.} The resulting instance of (\ref{eq:CTC}) satisfies
the assumptions of Theorem~\ref{thm:lintime} and therefore costs
$O(\omega^{6.5}n^{1.5}\log(1/\epsilon))$ time and $O(\omega^{4}n)$
memory to solve.

\textbf{Step 4.} The low-rank matrix completion algorithm~\cite[Algorithm~2]{sun2015decomposition}
makes $\ell\le n$ iterations, where each iteration performs $O(1)$
matrix-matrix operations over $\omega\times\omega$ dense matrices.
Its cost is therefore $O(\omega^{3}n)$ time and $O(\omega^{2}n)$
memory.\qed\end{proof}

\section{\label{sec:Auxillary-variables}Dualized Clique Tree Conversion with
auxiliary Variables}

Theorem~\ref{thm:lintime} bounds the cost of solving instances of
(\ref{eq:CTC}) that satisfy the sparsity assumption (\ref{eq:structA})
as near-linear time and linear memory. Instances of (\ref{eq:CTC})
that do not satisfy the sparsity assumption can be systematically
transformed into ones that do by introducing \emph{auxiliary variables}.
Let us illustrate this idea with an example.
\begin{example}[Path graph]
\label{exa:path}Given $(n+1)\times(n+1)$ symmetric tridiagonal
matrices $A\succ0$ and $C$ with $A[i,j]=C[i,j]=0$ for all $|i-j|>1$,
consider the Rayleigh quotient problem
\begin{equation}
\text{minimize }C\bullet X\qquad\text{ subject to }\qquad A\bullet X=1,\quad X\succeq0.\label{eq:exa_sdp}
\end{equation}
The associated sparsity graph is the \emph{path graph} on $n+1$ nodes,
and its elimination tree decomposition $\T=(\{J_{1},\ldots,J_{n}\},T)$
has index sets $J_{j}=\{j,j+1\}$ and parent pointer $p(j)=j+1$.
Applying clique tree conversion and vectorizing yields an instance
of (\ref{eq:vecCTC}) with
\begin{align*}
\A & =\begin{bmatrix}a_{1}^{T} & \cdots & a_{n}^{T}\end{bmatrix}, & \N & =\begin{bmatrix}e_{3}^{T} & -e_{1}^{T}\\
 & \ddots & \ddots\\
 &  & e_{3}^{T} & -e_{1}^{T}
\end{bmatrix}
\end{align*}
where $e_{j}$ is the $j$-th column of the $3\times3$ identity matrix,
and $a_{1},\ldots,a_{n}\in\R^{3}$ are appropriately chosen vectors.
The dualized Schur complement $\H=\D_{s}+\sigma\A^{T}\A+\sigma\N^{T}\N$
is fully dense, so dualized clique tree conversion (Algorithm~\ref{alg:dctc})
would have a complexity of at least cubic $n^{3}$ time and quadratic
$n^{2}$ memory. Instead, introducing $n-1$ auxiliary variables $u_{1},\ldots,u_{n-1}$
yields the following problem
\begin{alignat}{3}
 & \text{minimize }\quad & \sum_{j=1}^{n}c_{j}^{T}x_{j}\label{eq:auxpath}\\
 & \text{subject to } & a_{1}^{T}x_{1}-\begin{bmatrix}0 & 1\end{bmatrix}\begin{bmatrix}x_{2}\\
u_{2}
\end{bmatrix} & =b\nonumber \\
 &  & \begin{bmatrix}a_{i}^{T} & 1\end{bmatrix}\begin{bmatrix}x_{i}\\
u_{i}
\end{bmatrix}-\begin{bmatrix}0 & 1\end{bmatrix}\begin{bmatrix}x_{i+1}\\
u_{i+1}
\end{bmatrix} & =0 & \qquad\text{for all }i & \in\{2,\ldots n-1\}\nonumber \\
 &  & x_{1}\in\vector(\S_{+}^{2}),\;\begin{bmatrix}x_{j}\\
u_{j}
\end{bmatrix}\in\vector(\S_{+}^{2}) & \times\R & \text{for all }j & \in\{2,\ldots n\}\nonumber 
\end{alignat}
which does indeed satisfy the sparsity assumption (\ref{eq:structA})
of Theorem~\ref{thm:lintime}. In turn, solving (\ref{eq:auxpath})
using Steps~2-3 of Algorithm~\ref{alg:dctc} recovers an $\epsilon$-accurate
solution in $O(n^{1.5}\log\epsilon^{-1})$ time and $O(n)$ memory.\qed
\end{example}
For a constraint $A_{i}\bullet X=b_{i}$ in (\ref{eq:SDP}), we assume
without loss of generality\footnote{Since $T$ is connected, we can always find a connected subset $W'$
satisfying $W\subseteq W'\subseteq V(T)$ and replace $W$ by $W'$.} that the corresponding constraint in (\ref{eq:CTC}) is split over
a \emph{connected subtree} of $T$ induced by a subset of vertices
$W\subseteq V(T)$, as in
\begin{equation}
\sum_{j\in W}A_{i,j}\bullet X[J_{j},J_{j}]=b_{i},\qquad T_{W}\equiv(W,\;E(T))\text{ is connected}.\label{eq:Acond}
\end{equation}
Then, the coupled constraint (\ref{eq:Acond}) can be decoupled into
$|W|$ constraints, by introducing $|W|-1$ auxiliary variables, one
for each edge of the connected subtree $T_{W}$:
\begin{align}
A_{i,j}\bullet X[J_{j},J_{j}]+\sum_{k\in\ch(j)}u_{k} & =\begin{cases}
b_{i} & k\text{ is root of }T_{W},\\
u_{j} & \text{otherwise,}
\end{cases}\qquad\text{for all }j\in W.\label{eq:TC_constr}
\end{align}
It is easy to see that (\ref{eq:Acond}) and (\ref{eq:TC_constr})
are equivalent by applying Gaussian elimination on the auxiliary variables.
\begin{lemma}
The matrix $X$ satisfies (\ref{eq:Acond}) if and only if there exists
$\{u_{j}\}$ such that $X$ satisfies (\ref{eq:TC_constr}).
\end{lemma}
Repeating the splitting procedure for every constraint in (\ref{eq:CTC})
yields a problem of the form
\begin{alignat}{2}
 & \text{minimize } & c^{T}x,\label{eq:auxfull}\\
 & \text{subject to } & \sum_{j\in W_{i}}(\A_{i,j}x_{j}+\B_{i,j}u_{i,j}) & =\f_{i}\qquad\text{for all }i\in\{1,\ldots,m\}\nonumber \\
 &  & \sum_{j=1}^{\ell}\N_{i,j}x_{j} & =0\qquad\text{for all }i\in\{1,\ldots,\ell\}\nonumber \\
 &  & \begin{bmatrix}x_{j}\\{}
[u_{i,j}]_{i=1}^{m}
\end{bmatrix}\in\vector(\S_{+}^{|J_{j}|})\times & \R^{\gamma_{j}}\qquad\text{for all }j\in\{1,\ldots,\ell\}\nonumber 
\end{alignat}
where $W_{i}$ is induces the connected subtree associated with $i$-th
constraint, and $\gamma_{j}$ is the total number of auxiliary variables
added to each $j$-th variable block. When (\ref{eq:dualCTC-2}) is
dualized and solved using an interior-point method, the matrix $\H=[\H_{i,j}]_{i,j=1}^{\ell}$
satisfies $\H_{i,j}=0$ for every $(i,j)\notin E(T)$, so by repeating
the proof of Lemma~\ref{lem:normal}, the cost of solving the normal
equation is again linear time. Incorporating this within any off-the-self
interior-point method again yields a fast interior-point method.
\begin{lemma}
\label{lem:aux}Let $\T=(\{J_{1},\ldots,J_{\ell}\},T)$ be a tree
decomposition for the sparsity graph of $C,A_{1},\ldots,A_{m}\in\S^{n}$,
and convert the  corresponding instance of (\ref{eq:CTC}) into (\ref{eq:auxfull}).
Under Assumptions~\ref{ass:lin} \&~\ref{ass:slater}, there exists
an algorithm that computes an iterate $(x,y,s)\in\K\times\R^{p}\times\K_{*}$
satisfying (\ref{eq:eps_accurate}) in
\[
O((\omega^{2}+\gamma_{\max})^{3}\omega^{0.5}n^{1.5}\log\epsilon^{-1})\text{ time and }O((\omega^{2}+\gamma_{\max})^{2}n)\text{ memory,}
\]
where $\omega=1+\wid(\T)$ and $\gamma_{\max}=\max_{j}\gamma_{j}$
is the maximum number of auxiliary variables added to a single variable
block.
\end{lemma}
\begin{proof}
We repeat the proof of Theorem~\ref{thm:lintime}, but slightly modify
the linear time normal equation result in Lemma~\ref{lem:normal}.
Specifically, we repeat the proof of Lemma~\ref{lem:normal}, but
note that each block $\H_{i,j}$ of $\H$ is now order $\frac{1}{2}\omega(\omega+1)+\gamma_{\max}$,
so that factoring in (ii) now costs $O((\omega^{2}+\gamma_{\max})^{3}n)$
time and $O((\omega^{2}+\gamma_{\max})^{2}n)$ memory, and substituting
in (iii) costs $O((\omega^{2}+\gamma_{\max})^{2}n)$ time and memory.
After $O(\sqrt{\omega n}\log\epsilon^{-1})$ interior-point iterations,
we again arrive at an $\epsilon$-accurate and $\epsilon$-feasible
solution to (\ref{eq:CTC}). \qed
\end{proof}

\begin{algorithm}
\caption{\label{alg:dctcaux}Dualized clique tree conversion with auxiliary
variables}

\textbf{Input.} Data vector $b\in\R^{m}$, data matrices $C,A_{1},\ldots,A_{m}$,
and tree decomposition $\T=(\{J_{1},\ldots,J_{\ell}\},T)$ for the
sparsity graph of the data matrices.

\textbf{Output.} An $\epsilon$-accurate solution of (\ref{eq:SDP})
in factored form $X^{\star}=UU^{T}$, where $U\in\R^{n\times\omega}$
and $\omega=\max_{j}|J_{j}|$.

\textbf{Algorithm.}
\begin{enumerate}
\item (Conversion) Reformulate (\ref{eq:SDP}) into (\ref{eq:CTC}).
\item (auxiliary variables) Vectorize (\ref{eq:CTC}) into (\ref{eq:vecCTC}),
and separate into (\ref{eq:auxfull}) by rewriting each (\ref{eq:Acond})
as (\ref{eq:TC_constr}).
\item (Dualization) Dualize (\ref{eq:auxfull}) into (\ref{eq:dualCTC-2}).
\item (Solution) Solve (\ref{eq:dualCTC-2}) as an order-$\nu$ conic linear
program in standard form, using an interior-point method with $O(\sqrt{\nu}\log(1/\epsilon))$
iteration complexity. At each iteration of the interior-point method,
solve the normal equation using sparse Cholesky factorization and
the fill-reducing permutation $\Pi$ in Lemma~\ref{lem:normal}.
Obtain $\epsilon$-accurate solutions $X_{1}^{\star},\ldots,X_{\ell}^{\star}$.
\item (Recovery) Recover $X^{\star}=UU^{T}$ satisfying $X^{\star}[J_{i},J_{i}]=X_{i}^{\star}$
using the low-rank matrix completion algorithm~\cite[Algorithm~2]{sun2015decomposition}.
\end{enumerate}
\end{algorithm}
The complete end-to-end procedure for solving (\ref{eq:SDP}) using
the auxiliary variables method is summarized as Algorithm~\ref{alg:dctcaux}.
In the case of network flow semidefinite programs, the separating
in Step 2 can be performed in closed-form using the\emph{ induced
subtree} property of the tree decomposition~\cite{blair1993introduction}.
\begin{definition}[Induced subtrees]
Let $\T=(\{J_{1},\ldots,J_{\ell}\},T)$ be a tree decomposition.
We define $T_{k}$ as the connected subtree of $T$ induced by the
nodes that contain the element $k$, as in
\begin{align*}
V(T_{k}) & =\{j\in\{1,\ldots,\ell\}:k\in J_{j}\}, & E(T_{k}) & =E(T).
\end{align*}
\end{definition}
\begin{lemma}
\label{lem:network_split}Let $\T=(\{J_{1},\ldots,J_{\ell}\},T)$
be a tree decomposition for the graph $G$. For every $i\in V(G)$
and
\[
A=\alpha_{i}e_{i}e_{i}^{T}+\sum_{(i,j)\in E(G)}\alpha_{j}(e_{i}e_{j}^{T}+e_{j}e_{i}^{T}),
\]
there exists $A_{j}$ for $j\in V(T_{i})$ such that
\[
A\bullet X=\sum_{j\in V(T_{i})}A_{j}\bullet X[J_{j},J_{j}]\qquad\text{for all }X\in\S^{n}.
\]
\end{lemma}
\begin{proof}
We give an explicit construction. Iterate $j$ over the neighbors
$\mathrm{nei}(i)=\{j:(i,j)\in E(G)\}$ of $i$. By the edge cover
property of the tree decomposition, there exists $k\in\{1,\ldots,\ell\}$
satisfying $i,j\in J_{k}$. Moreover, $k\in V(T_{i})$ because $i\in J_{k}$.
Define $A_{k}$ to satisfy 
\[
A_{k}\bullet X[J_{k},J_{k}]=(\alpha_{i}/\deg_{i})X[i,i]+\alpha_{j}(X[i,j]+X[j,i]),
\]
where $\deg_{i}=|\mathrm{nei}(i)|$. \qed
\end{proof}

If each network flow constraint is split using according to Lemma~\ref{lem:network_split},
then the number of auxiliary variables needed to decouple the problem
can be bounded. This results in a proof of Theorem~\ref{thm:comp3}.

\begin{proof}[Theorem~\ref{thm:comp3}]By hypothesis, $\T=\{J_{1},\ldots,J_{\ell}\}$
is a tree decomposition for the sparsity graph of the data matrices
$C,A_{1},\ldots,A_{m}$, and each $A_{i}$ can be split according
to Lemma~\ref{lem:network_split} onto a connected subtree of $T$.
We proceed to solve (\ref{eq:SDP}) using Algorithm~\ref{alg:dctcaux}.
We perform Step 1 in closed-form, by splitting each $A_{i}$ in according
to Lemma~\ref{lem:network_split}. The cost of Steps 2 and 3 are
then bound as $\nnz(\A)+\nnz(\N)=O(\omega^{3}n)$ time and memory.
The cost of step 5 is also $O(\omega^{3}n)$ time and $O(\omega^{2}n)$
memory, using the same reasoning as the proof of Theorem~\ref{thm:comp2}.

To quantify the cost of Step 4, we must show that under the conditions
stated in the theorem, the maximum number of auxiliary variables added
to each variable block is bound $\gamma_{j}\le m_{k}\cdot\omega\cdot d_{\max}$.
We do this via the following line of reasoning:
\begin{itemize}
\item A single network flow constraint at vertex $k$ contributes $|\ch(j)|\le d_{\max}$
auxiliary variables to every $j$-th index set $J_{j}$ satisfying
$j\in V(T_{k})$. 
\item Having one network flow constraint at every $k\in\{1,\ldots,\ell\}$
contributes at most $\omega\cdot d_{\max}$ auxiliary variables to
every $j$-th clique $J_{j}$. This is because the set of $V(T_{k})$
for which $j\in V(T_{k})$ is exactly $J_{j}=\{\{1,\ldots,\ell\}:j\in V(T_{k})\}$,
and $|J_{j}|\le\omega$ by definition.
\item Having $m_{k}$ network flow constraints at each $k\in\{1,\ldots,\ell\}$
contributes at most $m_{k}\cdot\omega\cdot d_{\max}$ auxiliary variables
to every $j$-th clique $J_{j}$.
\end{itemize}
Finally, applying $\gamma_{j}\le m_{k}\cdot\omega\cdot d_{\max}$
to Lemma~\ref{lem:aux} yields the desired complexity figure, which
dominates the cost of the entire algorithm. \qed \end{proof}

\section{\label{sec:num-exp}Numerical Experiments}

Using the techniques described in this paper, we solve sparse semidefinite
programs posed on the 40 power system test cases in the MATPOWER suite~\cite{zimmerman2011matpower},
each with number of constraints $m$ comparable to $n$. The largest
two cases have $n=9241$ and $n=13659$, and are designed to accurately
represent the size and complexity of the European high voltage electricity
transmission network~\cite{josz2016ac}. In all of our trials below,
the accuracy of a primal-dual iterate $(X,y,S)$ is measured using
the DIMACS feasibility and duality gap metrics~\cite{mittelmann2003independent}
and stated as the number of accurate decimal digits:
\begin{align*}
\mathrm{pinf} & =-\log_{10}\left[\;\|\mathcal{A}(X)-b\|_{2}/(1+\|b\|_{2})\;\right],\\
\mathrm{dinf} & =-\log_{10}\left[\;\lambda_{\max}(\mathcal{A}^{T}(y)-C)/(1+\|C\|_{2})\;\right],\\
\mathrm{gap} & =-\log_{10}\left[\;(C\bullet X-b^{T}y)/(1+|C\bullet X|+|b^{T}y|)\;\right],
\end{align*}
where $\mathcal{A}(X)=[A_{i}\bullet X]_{i=1}^{m}$ and $\mathcal{A}^{T}(y)=\sum_{i=1}^{m}y_{i}A_{i}$.
We will frequently measure the overall number of accurate digits as
$L=\min\{\mathrm{gap},\mathrm{pinf},\mathrm{dinf}\}$. 

In our trials, we implement Algorithm~\ref{alg:dctc} and Algorithm~\ref{alg:dctcaux}
in MATLAB using a version of SeDuMi v1.32~\cite{sturm1999using}
that is modified to force a specific fill-reducing permutation during
symbolic factorization. The actual block topological ordering that
we force SeDuMi to use is a simple postordering of the elimination
tree. For comparison, we also implement both algorithms using the
standard off-the-shelf version of MOSEK v8.0.0.53~\cite{andersen2000mosek},
without forcing a specific fill-reducing permutation. The experiments
are performed on a Xeon 3.3 GHz quad-core CPU with 16 GB of RAM. 

\subsection{Elimination trees with small widths}

\begin{figure}
\begin{lstlisting}
E = C | A1 | A2 | A3 | A4;
p = amd(E); % fill-reducing ordering
[~,~,parT,~,R] = symbfact(E(p,p), 'sym', 'lower');
R(p,:) = R; % Reverse the ordering
for i = 1:n, J{i} = find(R(:,i)); end
\end{lstlisting}

\caption{\label{fig:MATLAB-code-snippet}MATLAB code for computing the tree
decomposition of a given sparsity graph. The code terminates with
tree decomposition $\protect\T=(\protect\J,T)$ in which the index
sets $\protect\J=\{J_{1},\ldots,J_{n}\}$ are stored as the cell array
\texttt{J}, and the tree $T$ is stored in terms of its parent pointer
\texttt{parT}.}
\end{figure}
\begin{table}
\caption{\label{tab:treedecomp}Tree decompositions for the 40 test power systems
under study: $|V(G)|$ - number of vertices; $|E(G)|$ - number of
edges; $|\protect\J|$ - number of bags in $\protect\T$; $\omega=1+\protect\wid(\protect\T)$
- computed clique number; ``Time'' - total computation time in seconds.}

\setlength\tabcolsep{3pt}
\resizebox{\textwidth}{!}{

\begin{tabular}{|c|cccccc|c|cccccc|}
\hline 
\# & Name & $|V(G)|$ & $|E(G)|$ & \color{blue} $|\J|$ & $\omega$ & Time & \# & Name & $|V(G)|$ & $|E(G)|$ & \color{blue} $|\J|$ & $\omega$ & Time\tabularnewline
\hline 
1 & case4gs & 4 & 4 & \color{blue} 2 & 3 & 0.171 & 21 & case1354pegase & 1354 & 1991 & \color{blue} 1287 & 13 & 0.155\tabularnewline
2 & case5 & 5 & 6 & \color{blue} 3 & 3 & 0.030 & 22 & case1888rte & 1888 & 2531 & \color{blue} 1816 & 13 & 0.213\tabularnewline
3 & case6ww & 6 & 11 & \color{blue} 2 & 4 & 0.014 & 23 & case1951rte & 1951 & 2596 & \color{blue} 1879 & 14 & 0.219\tabularnewline
4 & case9 & 9 & 9 & \color{blue} 7 & 3 & 0.027 & 24 & case2383wp & 2383 & 2896 & \color{blue} 2312 & 25 & 0.278\tabularnewline
5 & case9Q & 9 & 9 & \color{blue} 7 & 3 & 0.011 & 25 & case2736sp & 2736 & 3504 & \color{blue} 2652 & 25 & 0.310\tabularnewline
6 & case9target & 9 & 9 & \color{blue} 7 & 3 & 0.002 & 26 & case2737sop & 2737 & 3506 & \color{blue} 2653 & 24 & 0.317\tabularnewline
7 & case14 & 14 & 20 & \color{blue} 12 & 3 & 0.006 & 27 & case2746wop & 2746 & 3514 & \color{blue} 2653 & 26 & 0.314\tabularnewline
8 & case24\_ieee\_rts & 24 & 38 & \color{blue} 20 & 5 & 0.017 & 28 & case2746wp & 2746 & 3514 & \color{blue} 2659 & 24 & 0.312\tabularnewline
9 & case30 & 30 & 41 & \color{blue} 26 & 4 & 0.005 & 29 & case2848rte & 2848 & 3776 & \color{blue} 2739 & 18 & 0.334\tabularnewline
10 & case30Q & 30 & 41 & \color{blue} 26 & 4 & 0.004 & 30 & case2868rte & 2868 & 3808 & \color{blue} 2763 & 17 & 0.323\tabularnewline
11 & case30pwl & 30 & 41 & \color{blue} 26 & 4 & 0.004 & 31 & case2869pegase & 2869 & 4582 & \color{blue} 2700 & 15 & 0.317\tabularnewline
12 & case\_ieee30 & 30 & 41 & \color{blue} 26 & 4 & 0.004 & 32 & case3012wp & 3012 & 3572 & \color{blue} 2916 & 28 & 0.344\tabularnewline
13 & case33bw & 33 & 37 & \color{blue} 32 & 2 & 0.005 & 33 & case3120sp & 3120 & 3693 & \color{blue} 3029 & 27 & 0.353\tabularnewline
14 & case39 & 39 & 46 & \color{blue} 34 & 4 & 0.005 & 34 & case3375wp & 3374 & 4161 & \color{blue} 3248 & 30 & 0.378\tabularnewline
15 & case57 & 57 & 80 & \color{blue} 52 & 6 & 0.010 & 35 & case6468rte & 6468 & 9000 & \color{blue} 6153 & 30 & 0.725\tabularnewline
16 & case89pegase & 89 & 210 & \color{blue} 77 & 12 & 0.011 & 36 & case6470rte & 6470 & 9005 & \color{blue} 6149 & 30 & 0.716\tabularnewline
17 & case145 & 145 & 453 & \color{blue} 111 & 5 & 0.018 & 37 & case6495rte & 6495 & 9019 & \color{blue} 6171 & 31 & 0.713\tabularnewline
18 & case118 & 118 & 186 & \color{blue} 108 & 11 & 0.020 & 38 & case6515rte & 6515 & 9037 & \color{blue} 6193 & 31 & 0.716\tabularnewline
19 & case\_illinois200 & 200 & 245 & \color{blue} 189 & 9 & 0.024 & 39 & case9241pegase & 9241 & 16049 & \color{blue} 8577 & 35 & 1.009\tabularnewline
20 & case300 & 300 & 411 & \color{blue} 278 & 7 & 0.035 & 40 & case13659pegase & 13659 & 20467 & \color{blue} 12997 & 35 & 1.520\tabularnewline
\hline 
\end{tabular}

}
\end{table}
We begin by computing tree decompositions using MATLAB's internal
approximate minimum degree heuristic (due to Amestoy, Davis and Duff~\cite{amestoy2004algorithm}).
A simplified version of our code is shown as the snippet in Figure~\ref{fig:MATLAB-code-snippet}.
(Our actual code uses Algorithm~4.1 in~\cite{vandenberghe2015chordal}
to reduce the computed elimination tree to the \emph{supernodal} elimination
tree, for a slight reduction in the number of index sets $\ell$.)
Table~\ref{tab:treedecomp} gives the details and timings for the
40 power system graphs from the MATPOWER suite~\cite{zimmerman2011matpower}.
As shown, we compute tree decompositions with $\wid(\T)\le34$ in
less than $2$ seconds. In practice, the bottleneck of the preprocessing
step is not the tree decomposition, but the constraint splitting step
in Algorithm~\ref{alg:setcover-tree}.

\subsection{MAX 3-CUT and Lovasz Theta}

\begin{figure}
\subfloat[]{\includegraphics[width=0.49\columnwidth]{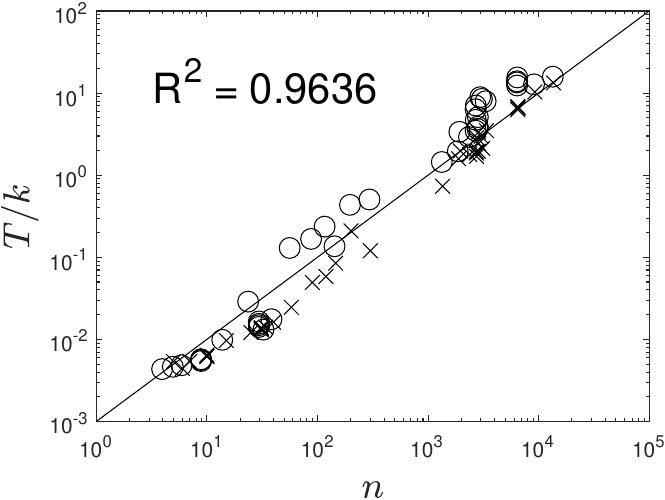}

}\subfloat[]{\includegraphics[width=0.49\columnwidth]{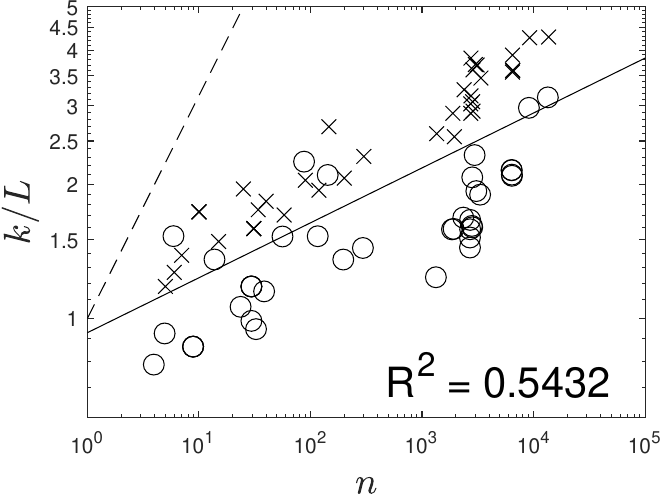}

}

\caption{\label{fig:results}SeDuMi Timings for MAX $3$-CUT ($\circ$) and
Lovasz Theta ($\times$) problems: (a) Time per iteration, with regression
$T/k=10^{-3}n$; (b) Iterations per decimal digit of accuracy, with
(solid) regression $k/L=0.929n^{0.123}$ and (dashed) bound $k/L=\sqrt{n}$.}

\vspace*{\floatsep}

\subfloat[]{\includegraphics[width=0.49\columnwidth]{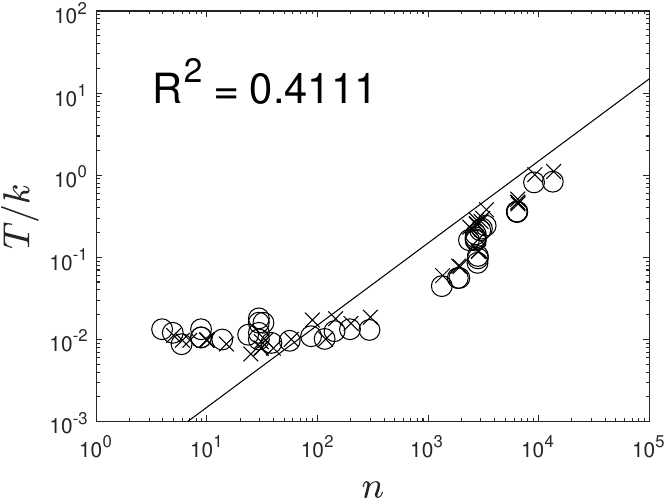}

}\subfloat[]{\includegraphics[width=0.49\columnwidth]{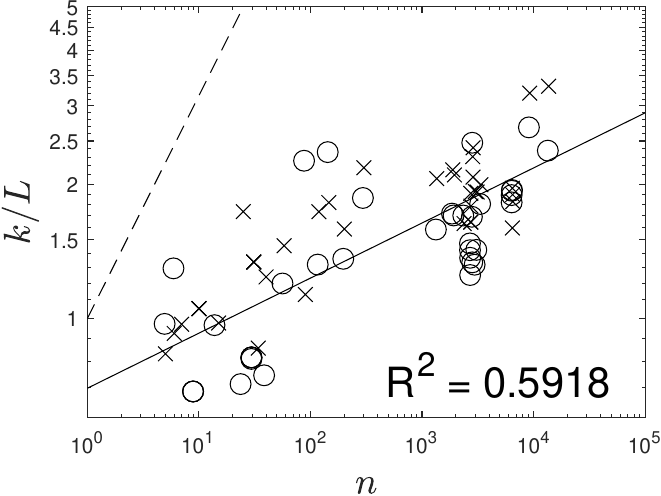}

}

\caption{\label{fig:resultsMOSEK}MOSEK Timings for MAX $3$-CUT ($\circ$)
and Lovasz Theta ($\times$) problems: (a) Time per iteration, with
regression $T/k=1.488\times10^{-4}n$; (b) Iterations per decimal
digit of accuracy, with (solid) regression $k/L=0.697n^{0.123}$ and
(dashed) bound $k/L=\sqrt{n}$.}
\end{figure}
We begin by considering the MAX 3-CUT and Lovasz Theta problems, which
are partially separable by default, and hence have solution complexities
of $O(n^{1.5})$ time and $O(n)$ memory. For each of the 40 test
cases, we use the MATPOWER function \texttt{makeYbus} to generate
the bus admittance matrix $Y_{bus}=[Y_{i,j}]_{i,j=1}^{n},$ and symmetrize
to yield $Y_{abs}=\frac{1}{2}[|Y_{i,j}|+|Y_{j,i}|]_{i,j=1}^{n}$.
We view this matrix as the weighted adjacency matrix for the system
graph. For MAX 3-CUT, we define the weighted Laplacian matrix $C=\diag(Y_{abs}\one)-Y_{abs}$,
and set up problem (\ref{eq:MkC}). For Lovasz Theta, we extract the
location of the graph edges from $Y_{abs}$ and set up (\ref{eq:LTp}). 

First, we use Algorithm~\ref{alg:dctc} with the modified version
of SeDuMi to solve the 80 instances of (\ref{eq:SDP}). Of the 80
instances considered, 79 solved to $L\ge5$ digits in $k\le23$ iterations
and $T\le306$ seconds; the largest instance solved to $L=4.48$.
Table~\ref{tab:maxcut} shows the accuracy and timing details for
the 20 largest problems solved. Figure~\ref{fig:results}a plots
$T/k$, the mean time taken per-iteration. As we guaranteed in Lemma~\ref{lem:lin_indep},
the per-iteration time is linear with respect to $n$. A log-log regression
yields $T/k=10^{-3}n$, with $R^{2}=0.9636$. Figure~\ref{fig:results}b
plots $k/L$, the number of iterations to a factor-of-ten error reduction.
We see that SeDuMi's guaranteed iteration complexity $k=O(\sqrt{n}\log\epsilon^{-1})=O(\sqrt{n}L)$
is a significant over-estimate; a log-log regression yields $k/L=0.929n^{0.123}\approx n^{1/8}$,
with $R^{2}=0.5432$. Combined, the data suggests an actual time complexity
of $T\approx10^{-3}n^{1.1}L$.

Next, we use Algorithm~\ref{alg:dctc} alongside the off-the-shelf
version of MOSEK to solve the 80 same instances. It turns out that
MOSEK is both more accurate than SeDuMi, as well as a factor of 5-10
faster. It manages to solve all 80 instances to $L\ge6$ digits in
$k\le21$ iterations and $T\le24$ seconds. Table~\ref{tab:theta}
shows the accuracy and timing details for the 20 largest problems
solved. Figure~\ref{fig:resultsMOSEK}a plots $T/k$, the mean time
taken per-iteration. Despite not forcing the use of a block topological
ordering, MOSEK nevertheless attains an approximately linear per-iteration
cost. Figure~\ref{fig:resultsMOSEK}b plots $k/L$, the number of
iterations to a factor-of-ten error reduction. Again, we see that
MOSEK's guaranteed iteration complexity $k=O(\sqrt{n}\log\epsilon^{-1})=O(\sqrt{n}L)$
is a significant over-estimate. A log-log regression yields an empirical
time complexity of $T\approx10^{-4}n^{1.12}L$, which is very close
to being linear-time.

\begin{table}
\caption{\label{tab:maxcut}Accuracy (in decimal digits) and timing (in seconds)
for 20 largest MAX $3$-CUT problems: $n$ - order of matrix variable;
$m$ - number of constraints; ``Pre-proc'' - post-processing time;
``gap'' - duality gap; ``pinf'' - primal infeasibility; ``dinf''
- dual infeasibility; $k$ - number of interior-point iterations;
$T$ - total interior-point time; ``Post-proc'' - post-processing
time.}

\resizebox{\textwidth}{!}{

\begin{tabular}{|c|cc|c|ccccc|ccccc|c|}
\hline 
 &  &  & Pre- & \multicolumn{5}{c|}{MOSEK} & \multicolumn{5}{c|}{SeDuMi} & Post-\tabularnewline
\cline{5-14} \cline{6-14} \cline{7-14} \cline{8-14} \cline{9-14} \cline{10-14} \cline{11-14} \cline{12-14} \cline{13-14} \cline{14-14} 
\#  & $n$ & $m$ & proc & gap & pinf & dinf & $k$ & $T$ & gap & pinf & dinf & $k$ & $T$ & proc.\tabularnewline
\hline 
21 & 1354 & 3064 & 1.1 & 9.6 & 8.9 & 9.1 & 14 & 0.6 & 11.6 & 7.5 & 9.7 & 12 & 17.0 & 0.1\tabularnewline
22 & 1888 & 4196 & 1.5 & 8.9 & 8.2 & 8.4 & 14 & 0.8 & 8.2 & 7.1 & 9.4 & 13 & 24.9 & 0.2\tabularnewline
23 & 1951 & 4326 & 1.6 & 8.9 & 8.3 & 8.4 & 14 & 0.8 & 8.9 & 7.3 & 10.1 & 14 & 46.2 & 0.2\tabularnewline
24 & 2383 & 5269 & 2.1 & 9.0 & 8.3 & 8.4 & 14 & 2.2 & 7.8 & 7.3 & 8.5 & 13 & 37.1 & 0.4\tabularnewline
25 & 2736 & 5999 & 2.4 & 8.8 & 8.2 & 8.3 & 12 & 2.0 & 12.0 & 7.5 & 10.6 & 16 & 99.6 & 0.4\tabularnewline
26 & 2737 & 6000 & 2.4 & 9.0 & 8.5 & 8.5 & 12 & 1.9 & 11.4 & 6.8 & 9.7 & 14 & 47.7 & 0.4\tabularnewline
27 & 2746 & 6045 & 2.4 & 11.3 & 10.4 & 10.8 & 13 & 2.4 & 11.3 & 6.4 & 9.5 & 15 & 69.3 & 0.4\tabularnewline
28 & 2746 & 6019 & 2.4 & 10.9 & 10.3 & 10.3 & 14 & 2.2 & 11.9 & 7.1 & 10.3 & 17 & 117.3 & 0.4\tabularnewline
29 & 2848 & 6290 & 2.5 & 8.9 & 8.3 & 8.4 & 14 & 1.2 & 8.1 & 6.9 & 9.4 & 13 & 46.7 & 0.4\tabularnewline
30 & 2868 & 6339 & 2.6 & 10.4 & 9.8 & 9.9 & 13 & 1.2 & 8.2 & 6.9 & 9.5 & 13 & 49.5 & 0.4\tabularnewline
31 & 2869 & 6837 & 2.7 & 8.7 & 8.1 & 8.2 & 20 & 2.0 & 9.6 & 5.2 & 8.2 & 17 & 84.7 & 0.5\tabularnewline
32 & 3012 & 6578 & 2.7 & 9.8 & 9.1 & 9.3 & 12 & 2.5 & 7.8 & 7.3 & 10.1 & 18 & 157.1 & 0.5\tabularnewline
33 & 3120 & 6804 & 2.8 & 9.3 & 8.5 & 8.7 & 12 & 2.6 & 11.7 & 7.7 & 10.4 & 20 & 166.4 & 0.5\tabularnewline
34 & 3374 & 7442 & 3.2 & 9.1 & 8.3 & 8.5 & 15 & 3.6 & 10.0 & 5.8 & 8.5 & 16 & 124.7 & 0.6\tabularnewline
35 & 6468 & 14533 & 7.6 & 9.2 & 8.5 & 8.7 & 16 & 5.6 & 9.4 & 4.9 & 7.5 & 16 & 210.9 & 1.6\tabularnewline
36 & 6470 & 14536 & 7.6 & 9.7 & 8.8 & 9.2 & 16 & 5.6 & 9.4 & 5.0 & 7.5 & 16 & 218.2 & 1.6\tabularnewline
37 & 6495 & 14579 & 7.7 & 9.0 & 8.3 & 8.5 & 16 & 5.6 & 9.4 & 4.7 & 7.6 & 16 & 193.8 & 1.6\tabularnewline
38 & 6515 & 14619 & 7.7 & 9.0 & 8.3 & 8.5 & 16 & 5.6 & 9.8 & 5.3 & 8.2 & 17 & 257.8 & 1.6\tabularnewline
39 & 9241 & 23448 & 14.0 & 9.2 & 8.2 & 8.7 & 22 & 17.6 & 5.1 & 4.1 & 6.6 & 15 & 187.5 & 3.5\tabularnewline
40 & 13659 & 32284 & 23.5 & 9.2 & 8.4 & 8.7 & 20 & 16.3 & 4.5 & 3.8 & 6.0 & 14 & 216.9 & 6.0\tabularnewline
\hline 
\end{tabular}

}
\vspace*{\floatsep}

\caption{\label{tab:theta}Accuracy and Timing for 20 largest Lovasz Theta
problems. }

\resizebox{\textwidth}{!}{

\begin{tabular}{|c|cc|c|ccccc|ccccc|c|}
\hline 
 &  &  & Pre- & \multicolumn{5}{c|}{MOSEK} & \multicolumn{5}{c|}{SeDuMi} & Post-\tabularnewline
\cline{5-14} \cline{6-14} \cline{7-14} \cline{8-14} \cline{9-14} \cline{10-14} \cline{11-14} \cline{12-14} \cline{13-14} \cline{14-14} 
\#  & $n$ & $m$ & proc & gap & pinf & dinf & $k$ & $T$ & gap & pinf & dinf & $k$ & $T$ & proc.\tabularnewline
\hline 
21 & 1355 & 1711 & 0.8 & 11.6 & 8.3 & 8.8 & 17 & 1.0 & 6.4 & 5.4 & 6.2 & 16 & 11.7 & 0.2\tabularnewline
22 & 1889 & 2309 & 1.2 & 11.4 & 7.9 & 8.4 & 17 & 1.3 & 5.9 & 5.1 & 6.7 & 17 & 27.0 & 0.3\tabularnewline
23 & 1952 & 2376 & 1.2 & 11.2 & 7.6 & 8.1 & 16 & 1.2 & 6.3 & 5.5 & 6.9 & 16 & 31.6 & 0.3\tabularnewline
24 & 2384 & 2887 & 1.6 & 12.2 & 8.6 & 9.1 & 14 & 3.3 & 5.8 & 5.1 & 6.6 & 19 & 33.9 & 0.4\tabularnewline
25 & 2737 & 3264 & 1.8 & 11.4 & 7.9 & 8.4 & 13 & 3.4 & 6.6 & 5.6 & 6.8 & 19 & 36.9 & 0.5\tabularnewline
26 & 2738 & 3264 & 1.8 & 10.9 & 7.3 & 7.8 & 14 & 3.5 & 7.4 & 5.8 & 6.3 & 19 & 35.3 & 0.5\tabularnewline
27 & 2747 & 3300 & 1.8 & 13.2 & 9.1 & 9.8 & 15 & 4.3 & 5.2 & 4.7 & 6.7 & 20 & 57.6 & 0.6\tabularnewline
28 & 2747 & 3274 & 1.9 & 11.4 & 7.8 & 8.4 & 14 & 3.5 & 7.5 & 5.3 & 5.7 & 18 & 30.5 & 0.5\tabularnewline
29 & 2849 & 3443 & 1.9 & 11.1 & 7.4 & 7.8 & 17 & 2.0 & 8.5 & 5.1 & 5.6 & 17 & 33.3 & 0.5\tabularnewline
30 & 2869 & 3472 & 1.9 & 11.4 & 7.7 & 8.2 & 16 & 1.9 & 5.8 & 4.9 & 6.5 & 17 & 41.4 & 0.5\tabularnewline
31 & 2870 & 3969 & 2.0 & 11.1 & 7.5 & 7.9 & 18 & 3.0 & 6.1 & 5.2 & 6.1 & 22 & 74.0 & 0.6\tabularnewline
32 & 3013 & 3567 & 2.1 & 11.4 & 7.8 & 8.3 & 15 & 4.1 & 9.1 & 5.9 & 5.9 & 22 & 65.5 & 0.6\tabularnewline
33 & 3121 & 3685 & 2.2 & 14.6 & 8.9 & 10.5 & 17 & 5.1 & 8.8 & 5.6 & 5.7 & 21 & 44.5 & 0.7\tabularnewline
34 & 3375 & 4069 & 2.4 & 12.6 & 8.5 & 9.7 & 17 & 6.5 & 9.2 & 5.8 & 6.1 & 21 & 73.1 & 0.8\tabularnewline
35 & 6469 & 8066 & 5.5 & 13.7 & 8.8 & 9.4 & 14 & 7.2 & 5.1 & 4.8 & 6.9 & 20 & 137.0 & 2.1\tabularnewline
36 & 6471 & 8067 & 5.5 & 12.2 & 8.2 & 8.7 & 16 & 7.5 & 5.7 & 4.9 & 5.6 & 20 & 125.9 & 2.1\tabularnewline
37 & 6496 & 8085 & 5.6 & 12.9 & 9.0 & 9.4 & 17 & 7.9 & 5.7 & 4.9 & 5.6 & 20 & 131.4 & 2.0\tabularnewline
38 & 6516 & 8105 & 5.6 & 13.2 & 8.8 & 9.3 & 16 & 7.1 & 5.7 & 4.9 & 5.6 & 20 & 133.9 & 2.1\tabularnewline
39 & 9242 & 14208 & 10.0 & 10.4 & 6.2 & 6.7 & 20 & 20.3 & 6.2 & 4.7 & 5.4 & 23 & 237.2 & 4.6\tabularnewline
40 & 13660 & 18626 & 16.4 & 10.8 & 6.3 & 6.7 & 21 & 23.3 & 5.7 & 4.5 & 5.4 & 23 & 305.9 & 8.0\tabularnewline
\hline 
\end{tabular}

}
\end{table}

\subsection{Optimal power flow}

\begin{figure}
\subfloat[]{\includegraphics[width=0.49\columnwidth]{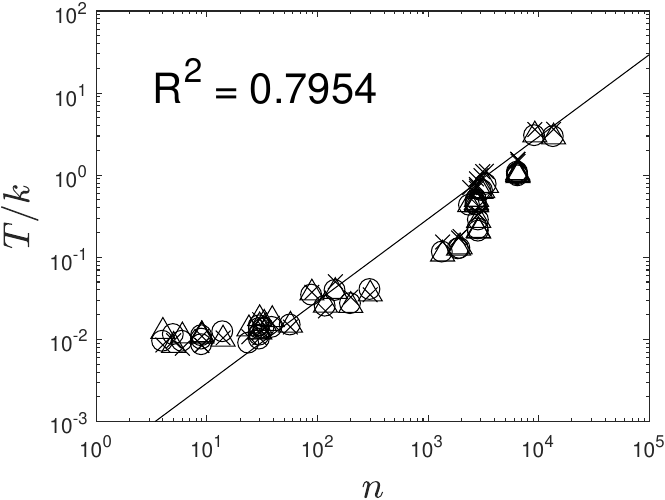}

}\subfloat[]{\includegraphics[width=0.49\columnwidth]{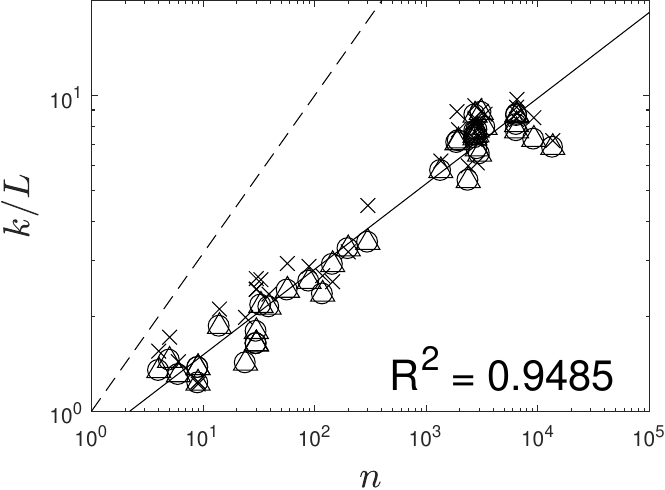}

}

\caption{\label{fig:results2}OPF problems solved using clique tree conversion
($\times$), dualized clique tree conversion ($\circ$) and dualized
clique tree conversion with auxiliary variables ($\triangle$): (a)
Time per iteration, with regression $T/k=2.931\times10^{-4}n$; (b)
Iterations per decimal digit of accuracy, with (solid) regression
$k/L=0.807n^{0.271}$ and (dashed) bound $k/L=\sqrt{n}$.}
\end{figure}
\begin{table}
\caption{\label{tab:opf}Accuracy (in decimal digits) and timing (in seconds)
for 20 largest OPF problems: $n$ - order of matrix variable; $m$
- number of constraints; ``Pre-proc'' - post-processing time; $L=\min\{\mathrm{gap},\mathrm{pinf},\mathrm{dinf}\}$
- accurate decimal digits; $k$ - number of interior-point iterations;
$T$ - total interior-point time; ``Post-proc'' - post-processing
time.}

\resizebox{\textwidth}{!}{

\begin{tabular}{|c|cc|c|ccc|ccc|ccc|c|}
\hline 
 &  &  & Pre- & \multicolumn{3}{c|}{CTC} & \multicolumn{3}{c|}{Dual CTC} & \multicolumn{3}{c|}{Dual CTC w/ aux} & Post-\tabularnewline
\cline{5-13} \cline{6-13} \cline{7-13} \cline{8-13} \cline{9-13} \cline{10-13} \cline{11-13} \cline{12-13} \cline{13-13} 
\#  & $n$ & $m$ & proc & $L$ & $k$ & $T$ & $L$ & $k$ & $T$ & $L$ & $k$ & $T$ & proc.\tabularnewline
\hline 
21 & 1354 & 4060 & 3.0 & 7.3 & 45 & 6.9 & 7.2 & 42 & 4.2 & 7.1 & 41 & 4.5 & 0.2\tabularnewline
22 & 1888 & 5662 & 4.1 & 7.2 & 64 & 11.0 & 6.9 & 48 & 6.3 & 6.8 & 48 & 6.2 & 0.3\tabularnewline
23 & 1951 & 5851 & 4.2 & 7.8 & 61 & 10.9 & 7.1 & 46 & 6.0 & 7.0 & 50 & 6.7 & 0.3\tabularnewline
24 & 2383 & 7147 & 5.9 & 7.2 & 43 & 30.4 & 6.9 & 38 & 16.6 & 6.9 & 37 & 16.2 & 0.4\tabularnewline
25 & 2736 & 8206 & 7.0 & 7.2 & 60 & 46.2 & 6.8 & 53 & 23.8 & 6.5 & 48 & 22.3 & 0.5\tabularnewline
26 & 2737 & 8209 & 6.8 & 7.1 & 66 & 45.7 & 6.7 & 57 & 24.5 & 6.9 & 53 & 23.1 & 0.5\tabularnewline
27 & 2746 & 8236 & 7.0 & 6.9 & 50 & 45.1 & 6.6 & 50 & 24.7 & 6.3 & 47 & 23.9 & 0.5\tabularnewline
28 & 2746 & 8236 & 6.9 & 7.1 & 60 & 44.2 & 6.7 & 56 & 25.2 & 6.9 & 60 & 26.7 & 0.6\tabularnewline
29 & 2848 & 8542 & 6.8 & 7.1 & 56 & 18.9 & 6.4 & 49 & 10.1 & 6.4 & 48 & 10.2 & 0.5\tabularnewline
30 & 2868 & 8602 & 6.8 & 7.4 & 56 & 18.8 & 6.6 & 51 & 10.6 & 6.7 & 52 & 10.8 & 0.5\tabularnewline
31 & 2869 & 8605 & 7.4 & 7.7 & 47 & 19.6 & 7.1 & 46 & 12.7 & 7.4 & 50 & 14.1 & 0.6\tabularnewline
32 & 3012 & 9034 & 7.9 & 7.0 & 55 & 54.5 & 6.1 & 54 & 31.6 & 6.9 & 45 & 28.5 & 0.6\tabularnewline
33 & 3120 & 9358 & 8.1 & 7.2 & 64 & 70.7 & 6.3 & 58 & 38.5 & 6.7 & 59 & 38.2 & 0.7\tabularnewline
34 & 3374 & 10120 & 8.9 & 7.1 & 62 & 69.3 & 6.6 & 56 & 39.8 & 6.6 & 52 & 39.0 & 0.7\tabularnewline
35 & 6468 & 19402 & 17.9 & 7.6 & 64 & 99.9 & 7.0 & 54 & 53.7 & 6.9 & 53 & 56.7 & 2.0\tabularnewline
36 & 6470 & 19408 & 18.0 & 7.4 & 68 & 106.1 & 6.8 & 57 & 56.3 & 6.9 & 56 & 57.2 & 2.0\tabularnewline
37 & 6495 & 19483 & 17.7 & 7.5 & 66 & 102.8 & 7.3 & 54 & 53.2 & 7.0 & 60 & 62.3 & 2.0\tabularnewline
38 & 6515 & 19543 & 17.7 & 7.2 & 70 & 103.4 & 6.8 & 54 & 54.7 & 6.8 & 59 & 58.1 & 2.0\tabularnewline
39 & 9241 & 27721 & 31.3 & 7.5 & 64 & 230.1 & 7.0 & 57 & 165.0 & 7.6 & 55 & 169.7 & 4.3\tabularnewline
40 & 13659 & 40975 & 47.9 & 6.8 & 49 & 177.4 & 7.9 & 48 & 154.6 & 7.9 & 54 & 157.4 & 7.7\tabularnewline
\hline 
\end{tabular}

}
\end{table}
We now solve instances of the OPF posed on the same 40 power systems
as mentioned above. Here, we use the MATPOWER function \texttt{makeYbus}
to generate the bus admittance matrix $Y_{bus}$, and then manually
generate each constraint matrix $A_{i}$ from $Y_{bus}$ using the
recipes described in~\cite{lavaei2012zero}. Specifically, we formulate
each OPF problem given the power flow case as follows: 
\begin{itemize}
\item Minimize the cost of generation. This is the sum of real-power injection
at each generator times \$1 per MW. 
\item Constrain all bus voltages to be from 95\% to 105\% of their nominal
values.
\item Constrain all load bus real-power and reactive-power values to be
from 95\% to 105\% of their nominal values.
\item Constrain all generator bus real-power and reactive-power values within
their power curve. The actual minimum and maximum real and reactive
power limits are obtained from the case description. 
\end{itemize}
We use three different algorithms based to solve the resulting semidefinite
program: 1) The original clique tree conversion of Fukuda and Nakata
et al.~\cite{fukuda2001exploiting,nakata2003exploiting} in Section~\ref{subsec:prelim_ctc};
2) Dualized clique tree conversion in Algorithm~\ref{alg:dctc};
3) Dualized clique tree conversion with auxiliary variables in Algorithm~\ref{alg:dctcaux}.
We solved all 40 problems using the three algorithms and MOSEK as
the internal interior-point solver. Table~\ref{tab:opf} shows the
accuracy and timing details for the 20 largest problems solved. All
three algorithms achieved near-linear time performance, solving each
problem instances to 7 digits of accuracy within 6 minutes. Upon closer
examination, we see that the two dualized algorithms are both about
a factor-of-two faster than the basic CTC method. Figure~\ref{fig:results2}
plots $T/k$, the mean time taken per-iteration, and $k/L$, the number
of iterations for a factor-of-ten error reduction, and their respective
log-log regressions. The data suggests an empirical time complexity
of $T\approx2.3\times10^{-4}n^{1.3}L$ over the three algorithms.

\section{Conclusion}

Clique tree conversion splits a large $n\times n$ semidefinite variable
$X\succeq0$ into up to $n$ smaller semidefinite variables $X_{j}\succeq0$,
coupled by a large number of overlap constraints. These overlap constraints
are a fundamental weakness of clique tree conversion, and can cause
highly sparse semidefinite program to be solved in as much as cubic
time and quadratic memory.

In this paper, we apply \emph{dualization} to clique tree decomposition.
Under a \emph{partially separable} sparsity assumption, we show that
the resulting normal equations have a block-sparsity pattern that
coincides with the adjacency matrix of a tree graph, so the per-iteration
time and memory complexity of an interior-point method is guaranteed
to be \emph{linear }with respect to $n$, the order of the matrix
variable $X$. Problems that do not satisfy the separable assumption
can be systematically separated by introducing auxiliary variables.
In the case of \emph{network flow} semidefinite programs, the number
of auxiliary variables can be bounded, so an interior-point method
again has a per-iteration time and memory complexity that is linear
with respect to $n$. 

Using these insights, we prove that the MAXCUT and MAX $k$-CUT relaxations,
the Lovasz Theta problem, and the AC optimal power flow relaxation
can all be solved with a guaranteed time and memory complexity that
is \emph{near-linear }with respect to $n$, assuming that a tree decomposition
with small width for the sparsity graph is known. Our numerical results
confirm an empirical time complexity that is \emph{linear} with respect
to $n$ on the MAX 3-CUT and Lovasz Theta relaxations.

\section*{Acknowledgments}

The authors are grateful to Daniel Bienstock, Salar Fattahi, Cédric
Josz, and Yi Ouyang for insightful discussions and helpful comments
on earlier versions of this manuscript. We thank Frank Permenter for
clarifications on various aspects of the homogeneous self-dual embedding
for SDPs. Finally, we thank the Associate Editor and Reviewer 2 for
meticulous and detailed comments that led to a significantly improved
paper. 

\appendix

\section{\label{sec:linindepslater}Linear independence and Slater's conditions}

In this section, we prove that (\ref{eq:CTC}) inherits the assumptions
of linear independence and Slater's conditions from (\ref{eq:SDP}).
We begin with two important technical lemmas.
\begin{lemma}
\label{lem:row-rank}The matrix $\N$ in (\ref{eq:Nmat}) has full
row rank, that is $\det(\N\N^{T})\ne0$.
\end{lemma}
\begin{proof}
We make $\N=[\N_{i,j}]_{i,j=1}^{\ell}$ upper-triangular by ordering
its blocks topologically on $T$: each nonempty block row $\N_{j}$
contains a nonzero block at $\N_{j,j}$ and a nonzero block at $\N_{j,p(j)}$
where the parent node $p(j)>j$ is ordered after $j$. Then, the claim
follows because each diagonal block $\N_{j,j}$ implements a surjection
and must therefore have full row-rank. \qed
\end{proof}

\begin{lemma}[Orthogonal complement]
\label{lem:orth}Let $d=\frac{1}{2}\sum_{j=1}^{\ell}|J_{j}|(|J_{j}|+1)$.
Implicitly define the $d\times\frac{1}{2}n(n+1)$ matrix $\P$ to
satisfy
\begin{align*}
\P\,\vector(X) & =[\vector(X[J_{1},J_{1}]);\ldots;\vector(X[J_{\ell},J_{\ell}])]\qquad\forall X\in\S^{n}.
\end{align*}
Then, (i) $\N\P=0$; (ii) every $x\in\R^{d}$ can be decomposed as
$x=\N^{T}u+\P v$. 
\end{lemma}
\begin{proof}
For each $x=[\vector(X_{j})]_{j=1}^{\ell}\in\R^{d}$, Theorem~\ref{thm:ctree}
says that there exists a $Z$ satisfying $\P\,\vector(Z)=x$ if and
only if $\N x=0$. Equivalently, $x\in\mathrm{span}(\P)$ if and only
if $x\perp\mathrm{span}(\N^{T})$. The ``only if'' part implies
(i), while the ``if'' part implies (ii).
\end{proof}

Define the $m\times\frac{1}{2}n(n+1)$ matrix $\M$ as the vectorization
of the linear constraints in (\ref{eq:SDP}), as in
\[
\M\,\vector(X)=\begin{bmatrix}\vector(A_{1})^{T}\\
\vdots\\
\vector(A_{m})^{T}
\end{bmatrix}\vector(X)=\begin{bmatrix}\vector(A_{1})^{T}\vector(X)\\
\vdots\\
\vector(A_{m})^{T}\vector(X)
\end{bmatrix}=\begin{bmatrix}A_{1}\bullet X\\
\vdots\\
A_{m}\bullet X
\end{bmatrix}.
\]
In reformulating (\ref{eq:SDP}) into (\ref{eq:CTC}), the splitting
conditions (\ref{eq:split_constr}) can be rewritten as the following
\begin{align}
c^{T}\P & =\vector(C)^{T}, & \A\P & =\M,\label{eq:split_constr2}
\end{align}
where $c=[\vector(C_{j})]_{j=1}^{\ell}$ and $\A=[\vector(A_{i,j})^{T}]_{i,j=1}^{m,\ell}$
are the data for the vectorized verison of (\ref{eq:CTC}).

\begin{proof}[Lemma~\ref{lem:lin_indep}]We will prove that
\[
\text{exists }[u;v]\ne0,\quad\A^{T}u+\N^{T}v=0\qquad\iff\qquad\text{exists }y\ne0,\quad\sum_{i=1}^{m}y_{i}A_{i}=0.
\]
\textbf{($\implies$)} We must have $u\ne0$, because $\N$ has full
row rank by Lemma~\ref{lem:row-rank}, and so $\A^{T}0+\N^{T}v=0$
if and only if $v=0$. Multiplying by $\P$ yields $\P^{T}(\A^{T}u+\N^{T}v)=\M^{T}u+0=0$
and so setting $y=u\ne0$ yields $\M^{T}y=0$.\textbf{ ($\impliedby$)
}We use Lemma~\ref{lem:orth} to decompose $\A^{T}y=\P z+\N^{T}v$.
If $\vector(\sum_{i}y_{i}A_{i})=\M^{T}y=\P^{T}\A^{T}y=0,$ then $\P^{T}\P z=\P^{T}\A^{T}y-\P^{T}\N^{T}v=0$
and so $\P z=0$. Setting $u=-y\ne0$ yields $\A^{T}u+\N^{T}v=0$.\qed \end{proof}

\begin{proof}[Lemma~\ref{lem:slaterp}]We will prove that 
\[
\text{exists }x\in\Int(\K),\quad\begin{bmatrix}\A\\
\N
\end{bmatrix}x=\begin{bmatrix}b\\
0
\end{bmatrix}\qquad\iff\qquad\text{exists }X\succeq0,\quad\M\,\vector(X)=b.
\]
Define the chordal completion $F$ as in (\ref{eq:chordal_compl}).
Observe that $\M\,\vector(X)=\M\,\vector(Z)$ holds for all pairs
of $P_{F}(X)=Z$, because each $A_{i}\in\S_{F}^{n}$ satisfies $A_{i}\bullet X=A_{i}\bullet P_{F}(X)$.
Additionally, the positive definite version of Theorem~\ref{thm:prim_sep}
is written
\begin{gather}
\text{exists }X\succ0,\quad P_{F}(X)=Z\qquad\iff\qquad\P\,\vector(Z)\in\Int(\K).\label{eq:prim_sep}
\end{gather}
This result was first established by Grone et al.~\cite{grone1984positive};
a succinct proof can be found in~\cite[Theorem~10.1]{vandenberghe2015chordal}.\textbf{
($\implies$)} For every $x$ satisfying $\N x=0$, there exists $Z$
such that $\P\,\vector(Z)=x$ due to Lemma~\ref{lem:orth}. If additionally
$x\in\Int(\K)$, then there exists $X\succ0$ satisfying $Z=P_{F}(X)$
due to (\ref{eq:prim_sep}). We can verify that $\M\,\vector(X)=\M\,\vector(Z)=\A\P\,\vector(Z)=\A x=b$.\textbf{
($\impliedby$)} For every $X\succ0,$ there exists $Z$ satisfying
$Z=P_{F}(X)$ and $\P\,\vector(Z)\in\Int(\K)$ due to (\ref{eq:prim_sep}).
Set $u=\P\,\vector(Z)$ and observe that $u\in\Int(\K)$ and $\N u=\N\P\,\vector(Z)=0$.
If additionally $\M\,\vector(X)=b$, then $\A u=\A\P\,\vector(Z)=\M\,\vector(Z)=b$.\qed \end{proof}

\begin{proof}[Lemma~\ref{lem:slaterp}]We will prove that
\[
\text{exists }u,v,\quad c-\A^{T}u-\N^{T}v\in\Int(\K_{*})\qquad\iff\qquad\text{exists }y,\quad C-\sum_{i}y_{i}A_{i}\succ0.
\]
Define the chordal completion $F$ as in (\ref{eq:chordal_compl}).
Theorem~\ref{thm:prim_sep} in (\ref{eq:prim_sep}) has a dual theorem
\begin{gather}
\text{exists }S\succ0,\quad S\in\S_{F}^{n}\qquad\iff\qquad\text{exists }h\in\Int(\K_{*}),\quad\vector(S)=\P^{T}h.\label{eq:dual_sep}
\end{gather}
This result readily follows from the positive semidefinite version
proved by Alger et al.~\cite{agler1988positive}; see also~\cite[Theorem~9.2]{vandenberghe2015chordal}.\textbf{
($\implies$)} For each $h=c-\A^{T}u-\N^{T}v$, define $S=C-\sum_{i}u_{i}A_{i}$
and observe that 
\[
\P^{T}h=\P^{T}(c-\A^{T}u-\N^{T}v)=\vector(C)-\M^{T}u-0=\vector(S).
\]
If additionally $h\in\K_{*}$, then $S\succ0$ due to (\ref{eq:dual_sep}).\textbf{
($\impliedby$)} For each $S=C-\sum_{i}y_{i}A_{i}\succ0$, there exists
an $h\in\Int(\K_{*})$ satisfying $\vector(S)=\P^{T}h$ due to (\ref{eq:dual_sep}).
We use Lemma~\ref{lem:orth} to decompose $h=\P u+\N^{T}v$. Given
that $\vector(S)=\P^{T}h=\P^{T}\P u+0$, we must actually have $\P u=c-\A^{T}y$
since $\P^{T}(c-\A^{T}y)=\vector(C)-\M^{T}y=\vector(S)$. Hence $h=c-\A^{T}y+\N^{T}v$
and $h\in\Int(\K_{*})$.\qed \end{proof}

\section{\label{sec:inequality}Extension to inequality constraints}

Consider the modifying the equality constraint in (\ref{eq:vecCTC})
into an inequality constraint, as in
\[
\text{minimize }c^{T}x\quad\text{subject to }\quad\A x\le b,\quad\N x=0,\quad x\in\K.
\]
The corresponding dualization reads
\[
\text{maximize }-c^{T}y\quad\text{subject to }\quad\begin{bmatrix}\A\\
\N\\
-I
\end{bmatrix}y+s=\begin{bmatrix}b\\
0\\
0
\end{bmatrix},\quad s\in\R_{+}^{m}\times\{0\}^{f}\times\K,
\]
where $m$ denotes the number of rows in $\A$ and $f$ now denotes
the number of rows in $\N$. Embedding the equality constraint into
a second-order cone, the associated normal equation takes the form
\[
\left(\D_{s}+\A^{T}\D_{l}\A+\N^{T}\D_{f}\N\right)\Delta y=r,
\]
where $\D_{s}$ and $\D_{f}$ are comparable as before in (\ref{eq:scal})
and (\ref{eq:scal_f}), and $\D_{l}$ is a diagonal matrix with positive
diagonal elements. This matrix has the same sparse-plus-rank-1 structure
as (\ref{eq:Hxr_dual}), and can therefore be solved using the same
rank-1 update 
\[
\Delta y=(\H+\q\q^{T})^{-1}r=\left(I-\frac{(\H^{-1}\q)\q^{T}}{1+\q^{T}(\H^{-1}\q)}\right)\H^{-1}r,
\]
where $\H$ and $\q$ now read
\begin{align*}
\H & =\D_{s}+\A^{T}\D_{l}\A+\sigma\N^{T}\N, & \q & =\N^{T}w.
\end{align*}
The matrix $\H$ has the same block sparsity graph as the tree graph
$T$, so we can evoke Lemma~\ref{lem:normal} to show that the cost
of computing $\Delta y$ is again $O(\omega^{6}n)$ time and $O(\omega^{4}n)$
memory. 

\section{\label{subsec:complexity}Interior-point method complexity analysis}

We solve the dualized problem (\ref{eq:dualCTC-2}) by solving its
extended homogeneous self-dual embedding\begin{subequations}\label{eq:hsd}
\begin{alignat}{2}
 & \text{min.}\quad & (\nu+1)\theta\\
 & \text{s.t.} & \begin{bmatrix}0 & +\M^{T} & -\c & -r_{d}\\
-\M & 0 & +\b & -r_{p}\\
+\c^{T} & -\b^{T} & 0 & -r_{c}\\
r_{d}^{T} & r_{p}^{T} & r_{c} & 0
\end{bmatrix}\begin{bmatrix}x\\
y\\
\tau\\
\theta
\end{bmatrix}+\begin{bmatrix}s\\
0\\
\kappa\\
0
\end{bmatrix} & =\begin{bmatrix}0\\
0\\
0\\
\nu+1
\end{bmatrix}\label{eq:hsdfeas}\\
 &  & x,s\in\C\quad\kappa,\tau & \ge0,\label{eq:hsdcone}
\end{alignat}
where the data is given in standard form
\begin{gather}
\M=\begin{bmatrix}0 & -\begin{bmatrix}\A^{T} & \N^{T}\end{bmatrix} & +I\end{bmatrix},\quad\c^{T}=\begin{bmatrix}0 & \begin{bmatrix}b^{T} & 0\end{bmatrix} & 0\end{bmatrix},\quad\b=-c,\\
\C=\{(x_{0},x_{1}\}\in\R^{p+1}:\|x_{1}\|\le x_{0}\}\times\K,
\end{gather}
and the residual vectors are defined
\begin{equation}
r_{d}=\one_{\C}-\c,\quad r_{p}=-\M\one_{\C}+\b,\quad r_{c}=1+\c^{T}\one_{\C}.
\end{equation}
\end{subequations}Here, $\nu$ is the order of the cone $\C$, and
$\one_{\C}$ is its identity element
\begin{align}
\nu & =1+|J_{1}|+\cdots+|J_{\ell}|, & \one_{\C} & =[1;0;\ldots;0;\vector(I_{|J_{1}|});\ldots;\vector(I_{|J_{\ell}|})].
\end{align}
Problem (\ref{eq:hsd}) has optimal value $\theta^{\star}=0$. Under
the primal-dual Slater's conditions (Assumption~\ref{ass:slater}),
an interior-point method is guaranteed to converge to an $\epsilon$-accurate
solution with $\tau>0$, and this yields an $\epsilon$-feasible and
$\epsilon$-accurate solution to the dualized problem (\ref{eq:dualCTC-2})
by rescaling $x/\tau$ and $y=y/\tau$ and $s=s/\tau$. The following
result is adapted from~\cite[Lemma~5.7.2]{de2000self} and \cite[Theorem~22.7]{vanderbei2015linear}.
\begin{lemma}[$\epsilon$-accurate and $\epsilon$-feasible]
\label{lem:eps}If $(x,y,s,\tau,\theta,\kappa)$ satisfies (\ref{eq:hsdfeas})
and (\ref{eq:hsdcone}) and 
\begin{align*}
\mu=\frac{x^{T}s+\tau\kappa}{\nu+1} & \le\epsilon, & \tau\kappa\ge & \gamma\mu
\end{align*}
for constants $\epsilon,\gamma>0$, then the rescaled point $(x/\tau,y/\tau,s/\tau)$
satisfies
\begin{align*}
\|\M(x/\tau)-\b\| & \le K\epsilon, & \|\M^{T}(y/\tau)+(s/\tau)-\c\| & \le K\epsilon, & \frac{(x/\tau)^{T}(s/\tau)}{\nu+1}\le K\epsilon
\end{align*}
where $K$ is a constant.
\end{lemma}
\begin{proof}
Note that (\ref{eq:hsdfeas}) implies $\mu=\theta$ and
\begin{align*}
\|\M(x/\tau)-\b\| & =\frac{\|r_{p}\|\mu}{\tau}, & \|\M^{T}(y/\tau)+(s/\tau)-\c\| & =\frac{\|r_{d}\|\mu}{\tau}, & \frac{(x/\tau)^{T}(s/\tau)}{\nu+1}=\frac{\mu}{\tau^{2}}.
\end{align*}
Hence, we obtain our desired result by upper-bounding $1/\tau$. Let
$(x^{\star},y^{\star},s^{\star},\tau^{\star},\theta^{\star},\kappa^{\star})$
be a solution of (\ref{eq:hsd}), and note that for every $(x,y,s,\tau,\theta,\kappa)$
satisfying (\ref{eq:hsdfeas}) and (\ref{eq:hsdcone}), we have the
following via the skew-symmetry of (\ref{eq:hsdfeas})
\[
(x-x^{\star})^{T}(s-s^{\star})+(\tau-\tau^{\star})(\kappa-\kappa^{\star})=0.
\]
Rearranging yields
\[
(\nu+1)\mu=x^{T}s+\tau\kappa=(x^{\star})^{T}s+x^{T}(s^{\star})+\tau\kappa^{\star}+\tau^{\star}\kappa\ge\tau^{\star}\kappa
\]
and hence
\[
\kappa\tau^{\star}\le\mu(\nu+1),\quad\tau\kappa\ge\gamma\mu\qquad\implies\qquad\tau\ge\frac{\gamma}{\nu+1}\tau^{\star}.
\]

If (\ref{eq:SDP}) satisfies the primal-dual Slater's condition, then
(\ref{eq:CTC}) also satisfies the primal-dual Slater's condition
(Lemmas~\ref{lem:slaterp} \&~\ref{lem:slaterd}). Therefore, the
vectorized version (\ref{eq:vecCTC}) of (\ref{eq:CTC}) attains a
solution $(\hat{x},\hat{y},\hat{s})$ with $\hat{x}^{T}\hat{s}=0$,
and the following
\[
x^{\star}=\tau^{\star}\begin{bmatrix}\|\hat{y}\|\\
\hat{y}\\
\hat{s}
\end{bmatrix},\quad y^{\star}=\tau^{\star}\hat{x},\quad s^{\star}=\tau^{\star}\begin{bmatrix}0\\
0\\
\hat{x}
\end{bmatrix},\quad\tau^{\star}=\frac{\nu+1}{\|\hat{y}\|+\one_{\K}^{T}\hat{s}+\one_{\K}^{T}\hat{x}+1},
\]
with $\theta^{\star}=\kappa^{\star}=0$ is a solution to (\ref{eq:hsd}).
This proves the following upper-bound 
\[
\frac{1}{\tau}\le K_{\tau}=\frac{1}{\gamma}\cdot\min_{\hat{x},\hat{y},\hat{s}}\{\|\hat{y}\|+\one_{\K}^{T}\hat{s}+\one_{\K}^{T}\hat{x}+1:(\hat{x},\hat{y},\hat{s})\text{ solves (\ref{eq:vecCTC}) with }\hat{x}^{T}\hat{s}=0\}.
\]
Setting $K=\max\{\|r_{p}\|K_{\tau},\|r_{d}\|K_{\tau},K_{\tau}^{2}\}$
yields our desired result. \qed
\end{proof}

We solve the homogeneous self-dual embedding (\ref{eq:hsd}) using
the short-step method of Nesterov and Todd~\cite[Algorithm~6.1]{nesterov1998primal}
(and also Sturm and S.~Zhang~\cite[Section~5.1]{sturm1999symmetric}),
noting that SeDuMi reduces to it in the worst case; see~\cite{sturm2002implementation}
and~\cite{sturm1997wide}. Beginning at the following strictly feasible,
perfectly centered point
\begin{gather}
\theta^{(0)}=\tau^{(0)}=\kappa^{(0)}=1,\quad y^{(0)}=0,\quad x^{(0)}=s^{(0)}=\one_{\C},\label{eq:hsdinit}
\end{gather}
with barrier parameter $\mu=1$, we take the following steps\begin{subequations}\label{eq:hsd_ipm}
\begin{align}
\mu^{+} & =\left(1-\frac{1}{15\sqrt{\nu+1}}\right)\cdot\frac{x^{T}s+\tau\kappa}{\nu+1},\\
(x^{+},y^{+},s^{+},\tau^{+},\theta^{+},\kappa^{+}) & =(x,y,s,\tau,\theta,\kappa)+(\Delta x,\Delta y,\Delta s,\Delta\tau,\Delta\theta,\Delta\kappa).\nonumber 
\end{align}
\end{subequations}along the search direction defined by the linear
system~\cite[Eqn. 9]{todd1998nesterov} \begin{subequations}\label{eq:hsd_nt}
\begin{align}
\begin{bmatrix}0 & +\M^{T} & -\c & -r_{p}\\
-\M & 0 & +\b & -r_{d}\\
+\c^{T} & -\b^{T} & 0 & -r_{c}\\
+r_{p}^{T} & +r_{d}^{T} & +r_{c} & 0
\end{bmatrix}\begin{bmatrix}\Delta x\\
\Delta y\\
\Delta\tau\\
\Delta\theta
\end{bmatrix}+\begin{bmatrix}\Delta s\\
0\\
\Delta\kappa\\
0
\end{bmatrix} & =0,\\
s+\Delta s+\nabla^{2}F(w)\Delta x+\mu^{+}\nabla F(x) & =0,\\
\kappa+\Delta\kappa+(\kappa/\tau)\Delta\tau-\mu^{+}\tau^{-1} & =0.
\end{align}
\end{subequations}Here, $F$ is the usual self-concordant barrier
function on $\C$
\begin{align}
F([x_{0};x_{1};\vector(X_{1});\ldots;\vector(X_{\ell})])= & -\log\left(\frac{1}{2}x_{0}^{2}-\frac{1}{2}x_{1}^{T}x_{1}\right)-\sum_{j=1}^{\ell}\log\det(X_{j})\label{eq:Fdef}
\end{align}
and $w\in\Int(\C)$ is the unique \emph{scaling point} satisfying
$\nabla^{2}F(w)x=s$, which can be computed from $x$ and $s$ in
closed-form. The following iteration bound is an immediate consequence
of~\cite[Theorem~6.4]{nesterov1998primal}; see also~\cite[Theorem~5.1]{sturm1999symmetric}.
\begin{lemma}[Short-Step Method]
\label{lem:short-step}The sequence in (\ref{eq:hsd_ipm}) arrives
at an iterate $(x,y,s,\tau,\theta,\kappa)$ satisfying the conditions
of Lemma~\ref{lem:eps} with $\gamma=9/10$ in at most $O(\sqrt{\nu}\log(1/\epsilon))$
iterations. 
\end{lemma}
The cost of each interior-point iteration is dominated by the cost
of computing the search direction in (\ref{eq:hsd_nt}). Using elementary
but tedious linear algebra, we can show that if\begin{subequations}\label{eq:newt}
\begin{equation}
(\M\D^{-1}\M^{T})\begin{bmatrix}v_{1} & v_{2} & v_{3}\end{bmatrix}=\begin{bmatrix}0 & -\b & r_{p}\end{bmatrix}-\M\D^{-1}\begin{bmatrix}d & \c & r_{d}\end{bmatrix}\label{eq:hsd_nrm}
\end{equation}
where $\D=\nabla^{2}F(w)$ and $d=-s-\mu^{+}\nabla F(x)$, and 
\begin{equation}
\begin{bmatrix}u_{1} & u_{2} & u_{3}\end{bmatrix}=\D^{-1}(\begin{bmatrix}d & c & r_{d}\end{bmatrix}+\M^{T}\begin{bmatrix}v_{1} & v_{2} & v_{3}\end{bmatrix}),\label{eq:hsd_back1}
\end{equation}
then
\begin{align}
\left(\begin{bmatrix}-\D_{0} & -r_{c}\\
r_{c} & 0
\end{bmatrix}-\begin{bmatrix}\c & r_{d}\\
-\b & r_{p}
\end{bmatrix}^{T}\begin{bmatrix}u_{2} & u_{3}\\
v_{2} & v_{3}
\end{bmatrix}\right)\begin{bmatrix}\Delta\tau\\
\Delta\theta
\end{bmatrix} & =\begin{bmatrix}-d_{0}\\
0
\end{bmatrix}-\begin{bmatrix}\c & r_{d}\\
-\b & r_{p}
\end{bmatrix}^{T}\begin{bmatrix}u_{1}\\
v_{1}
\end{bmatrix},\\
\begin{bmatrix}\Delta x\\
\Delta y
\end{bmatrix} & =\begin{bmatrix}u_{1}\\
v_{1}
\end{bmatrix}-\begin{bmatrix}u_{1} & u_{2}\\
v_{1} & v_{2}
\end{bmatrix}\begin{bmatrix}\Delta\tau\\
\Delta\theta
\end{bmatrix},\\
\Delta s & =d-\D\Delta x,\\
\Delta\kappa & =d_{0}-\D_{0}\Delta\tau,\label{eq:hsd_back2}
\end{align}
\end{subequations}where $\D_{0}=\kappa/\tau$ and $d_{0}=-\kappa+\mu^{+}\tau^{-1}$.
Hence, the cost of computing the search direction is dominated by
the cost of solving the normal equation for three different right-hand
sides. Here, the normal matrix is written

\[
\M\D^{-1}\M^{T}=\diag(W_{1}\otimes_{s}W_{1},\ldots,W_{\ell}\otimes_{s}W_{\ell})+\begin{bmatrix}\A\\
\N
\end{bmatrix}^{T}(w_{1}w_{1}^{T}+\sigma I)\begin{bmatrix}\A\\
\N
\end{bmatrix},
\]
where $\sigma=\frac{1}{2}(w_{0}^{2}-w_{1}^{T}w_{1})>0$ and $\otimes_{s}$
denotes the symmetric Kronecker product~\cite{alizadeh1998primal}
implicitly defined to satisfy
\[
(A\otimes_{s}B)\vector(X)=\frac{1}{2}\vector(AXB^{T}+BXA^{T})\qquad\text{for all }X=X^{T}.
\]
Under the hypothesis on $\A$ stated in Theorem~\ref{thm:lintime},
the normal matrix satisfies the assumptions of Lemma~\ref{lem:normal},
and can therefore be solved in linear $O(n)$ time and memory. 

\begin{proof}[Theorem~\ref{thm:lintime}]Combining Lemma~\ref{lem:eps}
and Lemma~\ref{lem:short-step} shows that the desired $\epsilon$-accurate,
$\epsilon$-feasible iterate is obtained after $O(\sqrt{\nu}\log(1/\epsilon))$
interior-point iterations. At each iteration we perform the following
steps: 1) compute the scaling point $w$; 2) solve the normal equation
(\ref{eq:hsd_nrm}) for three right-hand sides; 3) back-substitute
(\ref{eq:hsd_back1})-(\ref{eq:hsd_back2}) for the search direction
and take the step in (\ref{eq:hsd_ipm}). Note from the proof of Lemma~\ref{lem:normal}
that the matrix $[\A;\N]$ has at most $O(\omega^{2}n)$ rows under
Assumption~\ref{ass:lin}, and therefore $\nnz(\M)=O(\omega^{4}n)$
under the hypothesis of Theorem~\ref{thm:lintime}. Below, we show
that the cost of each step is bounded by $O(\omega^{6}n)$ time and
$O(\omega^{4}n)$ memory.

\textbf{Scaling point.} We partition $x=[x_{0};x_{1};\vector(X_{1});\ldots;\vector(X_{\ell})]$
and similarly for $s$. Then, the scaling point $w$ is given in closed-form~\cite[Section~5]{sturm2002implementation}
\begin{gather*}
\begin{bmatrix}u_{0}\\
u_{1}
\end{bmatrix}=2^{-3/4}\begin{bmatrix}1 & 1\\
-s_{1}/\|s_{1}\| & s_{1}/\|s_{1}\|
\end{bmatrix}\begin{bmatrix}(s_{0}-\|s_{1}\|)^{1/2}\\
(s_{0}+\|s_{1}\|)^{1/2}
\end{bmatrix},\quad\begin{bmatrix}v_{0}\\
v_{1}
\end{bmatrix}=\begin{bmatrix}u_{0} & u_{1}^{T}\\
u_{1} & \frac{1}{2}(u_{0}^{2}-u_{1}^{T}u_{1})I
\end{bmatrix}\begin{bmatrix}x_{0}\\
x_{1}
\end{bmatrix},\\
\begin{bmatrix}w_{0}\\
w_{1}
\end{bmatrix}=2^{-1/4}\begin{bmatrix}u_{0} & u_{1}^{T}\\
u_{1} & \frac{1}{2}(u_{0}^{2}-u_{1}^{T}u_{1})I
\end{bmatrix}\begin{bmatrix}1 & 1\\
-v_{1}/\|v_{1}\| & v_{1}/\|v_{1}\|
\end{bmatrix}\begin{bmatrix}(v_{0}-\|v_{1}\|)^{-1/2}\\
(v_{0}+\|v_{1}\|)^{-1/2}
\end{bmatrix},\\
W_{j}=S_{j}^{1/2}(S_{j}^{1/2}X_{j}S_{j}^{1/2})^{-1/2}S_{j}^{1/2}\qquad\text{for all }j\in\{1,\ldots,\ell\}.
\end{gather*}
Noting that $\nnz(w_{1})\le O(\omega^{2}n)$, $\ell\le n$ and each
$W_{j}$ is at most $\omega\times\omega$, the cost of forming $w=[w_{0};w_{1};\vector(W_{1});\ldots;\vector(W_{\ell})]$
is at most $O(\omega^{3}n)$ time and $O(\omega^{2}n)$ memory. Also,
since
\[
\D=\nabla^{2}F(w)=\diag\left(\begin{bmatrix}w_{0} & w_{1}^{T}\\
w_{1} & \frac{1}{2}(w_{0}^{2}-w_{1}^{T}w_{1})I
\end{bmatrix},W_{1}\otimes_{s}W_{1},\ldots,W_{\ell}\otimes_{s}W_{\ell}\right)^{-1},
\]
the cost of each matrix-vector product with $\D$ and $\D^{-1}$ is
also $O(\omega^{3}n)$ time and $O(\omega^{2}n)$ memory.

\textbf{Normal equation. }The cost of matrix-vector products with
$\M$ and $\M^{T}$ is $\nnz(\M)=O(\omega^{4}n)$ time and memory.
Using Lemma~\ref{lem:normal}, we form the right-hand sides and solve
the three normal equations in (\ref{eq:hsd_nrm}) in $O(\omega^{6}n)$
time and $O(\omega^{4}n)$ memory.

\textbf{Back-substitution.} The cost of back substituting (\ref{eq:hsd_back1})-(\ref{eq:hsd_back2})
and making the step (\ref{eq:hsd_ipm}) is dominated by matrix-vector
products with $\D$, $\D^{-1}$, $\M$, and $\M^{T}$ at $O(\omega^{4}n)$
time and memory.\qed \end{proof}

\bibliographystyle{spiebib}
\bibliography{chordConv2}

\end{document}